\documentclass[10pt,noinfoline]{article}
\include{frontmatter}
\begin{document}

%\begin{frontmatter}
%\title{The Asymptotic Distribution of Bootstrap Variance Estimator for Weighted Quantile}
%
%\runtitle{Bootstrap Variance Estimator for Weighted Quantile}
%\begin{aug}
%
%\author{\fnms{Jingchen} \snm{Liu}\thanksref{t1}}
%\and
%\author{\fnms{Xuan} \snm{Yang}}%\thanksref{t2}}
%%\author{\fnms{Jennifer} \snm{Hill}\thanksref{t3}}
%%\and
%%\author{\fnms{Yu-Sung} \snm{Su}\thanksref{t4}}
%
%\thankstext{t1}{Research supported in part by Institute of Education Sciences,
%U.S. Department of Education, through Grant R305D100017}
%%\thankstext{t2}{}
%%\thankstext{t3}{}
%%\thankstext{t4}{}
%\runauthor{Liu and Yang}
%\affiliation{Columbia University}
%\end{aug}
%
%\begin{abstract}
%In this paper, we provide the asymptotic distributions of the bootstrap variance estimators for quantile based on (weighted) empirical distributions. Under regularity conditions, we show that the bootstrap variance estimator is asymptotically normal and has relative standard deviation of order $n^{-1/4}$.
%\end{abstract}
%
%
%
%\begin{keyword}[class=AMS]
%\kwd[Primary ]{}
%\end{keyword}
%
%\begin{keyword}
%\kwd{Bootstrap, quantile, importance sampling}
%\end{keyword}
%\end{frontmatter}

\begin{frontmatter}
\title{Some Asymptotic Results of Gaussian Random Fields with Varying Mean Functions and the Associated Processes}

\runtitle{Some Asymptotic Results of Gaussian Random Fields}
\begin{aug}

\author{\fnms{Jingchen} \snm{Liu}\thanksref{t1}}
\and
\author{\fnms{Gongjun} \snm{Xu}}%\thanksref{t4}}

\thankstext{t1}{Research supported in part by Institute of Education Sciences, through Grant R305D100017, NSF CMMI-1069064, and NSF SES-1123698.}

%\thankstext{t2}{Research supported in part by Institute of Education Sciences, through Grant R305D090006.}
%\thankstext{t3}{Research supported in part by Institute of Education Sciences, through Grant R305D090006 and R305D100017.}

%\thankstext{t4}{}
\runauthor{Liu and Xu}
\affiliation{Columbia University}
\end{aug}

%
%\title{Some Asymptotic Results of Gaussian Random Fields with Varying Mean Functions and the Associated Processes}
%\author{Jingchen Liu and Gongjun Xu\\\\ Columbia University}
%
%\maketitle

\begin{abstract}
In this paper, we derive tail approximations of integrals of
exponential functions of Gaussian random fields with varying mean
functions and approximations of the associated point processes. This study is
motivated naturally by multiple applications such as hypothesis testing for spatial models and financial applications.
\end{abstract}


\end{frontmatter}

\section{Introduction}\label{SecIntro}

Gaussian random fields and multivariate Gaussian random vectors
constitute a cornerstone of statistics models in many disciplines,
such as physical oceanography and hydrology (\cite{AMR,Rubin}),
atmosphere study (\cite{DARO91}), geostatistics
(\cite{CHGE92,CHGE00}), astronomy (\cite{SGD93,GHVKP08}), and brain
imaging (\cite{SSSW, WT06,TaylorWorsley08, Taylor-Worsley-JASA}).
The difficulty very often lies in assessing the significance of the
test statistics due to the dependence structure induced by the
random field. In recent studies, closed form approximations of the
tail probabilities of supremum of random fields (the $p$-values)
have been studied intensively such as in
\cite{TaylorWorsley08,Taylor-Worsley-JASA, NSY08}. In this paper, we
develop asymptotic results of the integrals of exponential functions
of smooth Gaussian random fields with varying mean functions and the
associated point processes.

For concreteness, let $\{f(t): t\in T\}$ be a centered Gaussian
random field with unit variance and living on a $d$-dimensional domain $T\subset R^d$. For
every finite subset of $\{t_1,...,t_n\}\subset T$,
$(f(t_1),...,f(t_n))$ is a mean zero multivariate Gaussian random
vector. In addition, let $\mu(t)$ be a (deterministic) function. The
main quantity of interest is the probability
\begin{equation}\label{Tail}
  P\left(\int_T e^{\sigma f(t)+\mu(t)}dt > b\right),
\end{equation}
where $\sigma \in (0,\infty)$ is the scale factor.
In particular, we consider the asymptotic regime where $b$ tends to
infinity and develop closed form approximations of the above tail
probabilities. We further consider a doubly-stochastic Poisson process
$\{N(A): A\subset T\}$, with intensity $\{\lambda(t): t\in T\}$.
More specifically,  let $\log \lambda(t) = \sigma f(t)+\mu(t)$ be a continuous
Gaussian process. Conditional on $\{\lambda(t): t\in T\}$, $\{N(A): A\subset
T\}$ is an inhomogeneous Poisson process with intensity
$\lambda(t)$. Note that, conditional on the process $f(t)$,
$N(A)$ is a Poisson random variable with expectation $\int_A
e^{\sigma f(t)+\mu(t)}dt$. Then, we are interested in approximating the
tail probability
\begin{equation}\label{TailCount}
P(N(T) > b).
\end{equation}

The approximations of tail
probabilities in  \eqref{Tail} and \eqref{TailCount} are motivated by multiple
applications such as hypothesis testing for spatial models and financial
applications; see detailed discussions in Section \ref{SecApp}. In fact, \eqref{Tail} and \eqref{TailCount} are asymptotically the same. Therefore, the main result of this paper lies in developing approximations for \eqref{Tail}.

%In statistical applications, the Gaussian random field technique has been widely employed especially in spatial studies
In the statistics literature, closed form approximations of the tail
probabilities of Gaussian random fields have been widely employed
for the computation of $p$-values such as significance levels of the
scanning statistics (\cite{SY00, SSSW, NSY08, ATW09}). The works of
\cite{TaylorWorsley08,Taylor-Worsley-JASA,TaylorWorsley082} use
expected Euler characteristics of the excursion set as an
approximation and applied it to neuroimaging. \cite{RabinSiegmund97}
used saddlepoint approximation for the tail  of a smoothed Poisson
point process. Using a change of measure idea, \cite{NSY08} derived
the approximations for non-Gaussian fields in the context of a
likelihood-based hypothesis test.
%In this paper, we are particular interested in the asymptotic regime where $b$ tends to infinity.
In the probability literature, the extreme behavior of Gaussian random fields is also intensively studied. The results range from general bounds
to sharp asymptotic approximations. An incomplete list of works
includes {\cite{Hu90, LS70,MS70,ST74,Bor75,CIS,Berman85,LT91,TA96,Bor03}}.
A few lines of investigations on the supremum norm are given as
follows. Assuming locally stationary structure, the double-sum
method (\cite{Pit96}) provides the exact asymptotic approximation of
$\sup_T f(t)$ over a compact set $T$, which is allowed to grow as
the threshold tends to infinity. For almost surely at least twice
differentiable fields, \cite{Adl81,TTA05, AdlTay07} derive the
analytic form of the  expected Euler-Poicar\'e Characteristics of
the excursion set ($\chi(A_b)$) which serves as a good approximation
of the tail probability of the supremum. The tube method
(\cite{Sun93}) takes advantage of the Karhune-Lo\`eve expansion and
Weyl's formula. A recent related work along this line is given by
\cite{NSY08}. The Rice method (\cite{AW05,AW08,AW09}) provides an
implicit description of $\sup_T f(t)$. The discussions also go
beyond the Gaussian fields. For instance, \cite{HPZ11} discusses the situations of Gaussian process with random variances. See also  \cite{AST09} for other discussions.
%The change-of-measure based computational methods are recently
%provided by \cite{ABL08,ABL09,BLY}.

The analysis of integrals of non-linear functions of Gaussian
random fields is less developed compared with that of the supremum.
In the case that $f(t)$ is the Brownian motion, the distribution of
$\int_0^\infty e^{f(t)}dt$ is discussed by {\cite{Yor92,Duf01}}. For
 smooth and homogeneous Gaussian random fields,
the tail approximation of $\int_T e^{f(t)}dt$ is given by
{\cite{Liu10}} using a technique similar to the double-sum method.
The result and technique in {\cite{Liu10}} are restricted to the
homogenous fields (with a constant mean). In statistical analysis, however, allowing spatially varying mean functions
is usually very important especially in presence of
spatially varying covariates. Meanwhile, developing sharp
asymptotic approximations for random fields with spatially varying
means is a much more complicated and more difficult problem. The current work substantially generalizes the result of \cite{Liu10} and is applicable to more practical settings such as the presence of spatially varying covariates.

The contribution of this paper is to develop asymptotic
approximations of the probabilities as in \eqref{Tail} and
\eqref{TailCount} by introducing a change-of-measure technique.
This change of measure was first proposed by \cite{NSY08} to derive tail asymptotics of supremum of non-Gaussian random fields.
This technique substantially simplifies the analysis (though the
derivations are still complicated) and may potentially lead to
efficient importance sampling algorithms to numerically compute \eqref{Tail} and
\eqref{TailCount}; see {\cite{SIE76, ABL08, ABL09,BLY}}
for a connection between the change of measure and efficient computation
of tail probabilities. In addition, without too many modifications,
one can foresee that the proposed change of measure can be adapted
to certain non-Gaussian random fields such as those in
\cite{NSY08}, which uses a change-of-measure technique to develop
the approximations of suprema of non-Gaussian random fields with
functional expansions.

%The results presented in this paper also have both theoretical and practical implications. First, the results imply that the conditional distribution of the random field given a large value of the integration is approximately its distribution under the change of measure. This qualitatively results can be made quantitative (e.g. in terms of convergence in a total variation sense) by further investigate the higher moments of the Radon-Nikodym derivative under the change of measure. The conditional distribution is of interest, for instance, it helps to answer the question what the major cause is for the fact that there are unusually many positive cancer cases in a region. Second,

%\jc{implication of degree of independence between across different domain of the field.}

The organization of the rest of this paper is as follows. In Section
\ref{SecApp}, we present several applications of the current
study. The main results are given in Section \ref{SecMain} with
proofs provided in Section \ref{SecProof}. Some useful
lemmas are stated in Section \ref{SecLem}. A simulation study and technical proofs of several lemmas are provided as supplemental article \cite{LXsuppA}.

\section{Applications}\label{SecApp}

The integrals of exponential functions of random fields play an
important role in many applications. We present a few of
them in this section.
%For some of the applications, the results in this
%paper are directly applicable; for the others, more extended results
%are needed.

\subsection{Hypothesis testing}

\paragraph{Hypothesis testing for point processes}
Consider the doubly-stochastic Poisson process $N(\cdot)$ with intensity $\lambda(t)$ as defined in the introduction. The mean function of the log-intensity $\mu(t)$ is typically modeled as a linear combination of the observed spatially varying covariates, that is,
\begin{equation}\label{mu}
  \mu(t) = \xx^\top (t)\beta,
\end{equation} where $\xx(t)= (x_1(t),...,x_p(t))^\top$.
The Gaussian process $f(t)$ is then employed to build in a spatial
dependence structure by letting $\log\lambda(t)=f(t)+\mu(t)$. This
modeling approach has been widely used in the literature. For instance,
\cite{ChLe95} considers the time series setting in which $T$ is a
one dimensional interval, $\mu(t)$ in \eqref{mu} is modeled as the observed
covariate process  and $f(t)$ is  an autoregressive
process. See \cite{DDW00,Camp94,Zeger88,COX55,COIS80} for more
examples of such kind. For applications, this approach has been used
in many disciplines, such as astronomy, epidemiology, geography,
ecology and material science. Particularly, in the epidemiological
study, this model is used to describe the spatial
distribution of positive (e.g. cancer) cases over a region $T$ and
the latent intensity process is used to account for the unobserved
factors that may affect the hazard.

Under the above assumptions, we consider the related hypothesis testing
problems admitting a (simple) null hypothesis that the point process has log-intensity
$f(t) + \xx^\top (t) \beta$ with $\beta$ and the covariance function known. The alternative
hypothesis could be any probability model under which the
distribution of $N(T)$ stochastically dominates the one under the
null hypothesis. The one-sided $p$-value is then given by
\begin{equation}\label{tailc}
  P(N(T)>b),
\end{equation}
where $b$ is the observed count. This is equivalent to testing that a region $T$ is of a higher hazard level than the typical (null) level.
%One simple example is that $H_0 : \mu(t) =0$ and $H_A: \int_T \mu(t)dt >0$, where $\mu(t) = \xx^\top (t) \beta$.

For concreteness, we consider one situation that is frequently encountered in epidemiology.  Let the process $N(\cdot)$ denote the spatial locations of positive asthma cases in a certain region $T$ such as the New York City metropolitan. We assume that $N(\cdot)$ admits the doubly-stochastic structure described previously. To keep the example simple, we only include one covaraite in the model, that is, $\log \lambda(t) = \beta_{0} + \beta_{1}x(t)$, where $x(t)$ is the pollution level at location $t$ and has been standardized so that $\int_{T} x(t) dt =0$. Suppose that $\beta_{1}$ is known or an accurate estimate of $\beta_{1}$ is available. One is interested in testing the simple hypothesis $H_{0}: \beta_{0} = \beta_{0}^{*}$ against $H_{1}: \beta_{0}> \beta_{0}^{*}$, where ${\beta^{*}_{0}}$ is the national-wise log-intensity. A $p$-value is given by $P(N(T)>b)$. In this case, it is necessary to consider a spatially varying mean of the log-intensity to account for the inhomogeneity given that regression coefficient $\beta_1$ for pollution level is nonzero under the null hypothesis.

For other instances, the spatially varying covariates are sometimes chosen to functions to reflect certain periodicity.
One such case study is discussed in \cite{Zeger88} and further in \cite{DDW00} under the time series setting. In that example,  positive cases of poliomyelitis in the U.S.A. for the years 1970 - 1983 were observed. The time-varying covariates $\xx(t)$ are set to be a linear trend and harmonics at periods of 6 and 12 months and more precisely
$$\mu(t) = \xx(t)^\top \beta, \qquad \xx(t) = (1,t,\cos(2\pi t/12), \sin(2\pi t/ 12),\cos(2\pi t/6), \sin(2\pi t/ 6) ),$$
where $t$ is in the unit of one month. Similar hypothesis testing problems to the asthma case may be considered. The coefficients for most terms are significantly non-zero.
Thus, spatially varying covariates are ubiquitous in practice, which results in an non-constant mean under the null hypothesis.

In order to apply the results in this paper to a composite null hypothesis, such as the case in which the covariance
function of $f(t)$ is unknown, we need the corresponding estimates for some
characteristics of the covariance function of $f$ (see Theorem
\ref{ThmG}). The uncertainty
 of these estimates will definitely
introduce additional difficulty of the $p$-value calculation. On the other hand, with
reasonably large sample size, the necessary parameters can be
estimated accurately. In addition, the approximations stated in
later theorems only consist of the derivatives  of the covariance function (equivalently spectral moments) and $\mu(t)$ at the
global maximum. Then, one can design estimators simply for these
distributional characteristics, which are much easier to estimate
than the entire covariance functional form. There are  extensive
discussions on the estimations of spectral moments both
parametrically and nonparametrically  such as in the textbook
\cite{stein}. This plug-in-estimate strategy is used by
\cite{Taylor-Worsley-JASA} to handle such a composite null
hypothesis combined with a closed form $p$-value approximation by
means of the expected Euler characteristics. In that paper, the authors estimate and plug in the estimate of the
Lipschitz-Killing curvature to the expected Euler characteristic
function to approximate the tail probability of the supremum of a
$t$-field. Given that the main focus of this paper is on developing the approximations for the tail probabilities, we do not pursue parameter estimation aspects.

\paragraph{Hypothesis testing for aggregated data} The tail probability of $\int e^{\mu(t)+f(t)}dt$ itself can also serve as a $p$-value. In environmental science, the ozone concentration fluctuation is typically modeled to follow a log-normal distribution. For instance, it is found that the hourly averaged zone concentration typically admits a log-normal distribution; see \cite{HD02}. We let $\log\lambda(t) = \mu(t) +f(t)$ be the log-concentration of ozone at location $t$. One is interested in testing whether a region $T$ has an unusually high ozone level, that is, $H_0: E(\log\lambda(t)) = \mu(t) $ and $H_A: E(\log \lambda (t)) > \lambda(t)$ for $t\in T$.
Similar to the previous motivating example, in a regression setting, one may model $\mu(t) = \beta_0 + \xx^\top (t) \beta$ and consider $H_0: \beta_0 = \beta_0^*$ and $H_1:\beta_0 > \beta_0^*$.
One may reject the null if the observed aggregated ozone level, $b$,  in region $T$ is too high and a $p$-value is given by
$$P\Big(\int_T \lambda (t) dt = \int _T e^{\mu(t) + f(t)}dt > b\Big).$$
A similar argument as that for the point process application applies for the necessity of  incorporating a non-constant mean function $\mu(t)$ to reflect spatial inhomogeneity such as spatially varying covariates and periodicity.

\subsection{Financial applications}
The integrals of exponential
functions also play an important role in financial applications. This is
more related to the applied probability literature. In asset
pricing, the asset price indexed by time $t$ is typically modeled as
an exponential function of a Gaussian process, that is, $S(t) =
e^{f(t)}$. For instance, the Black-Scholes-Merton formula
\cite{BlaSch73, fMER73a} assumes that the price follows a geometric Brownian motion.
 Then, the payoff of an Asian option is the function of the
averaged price $\int_0^T e^{f(t)}dt$ and \eqref{Tail} is the
probability of exercising an Asian call option.

In the portfolio risk analysis, consider a portfolio consisting of
$n$ assets $(S_1,...,S_n)$ each of which is associated with a weight
(e.g. number of shares) $(w_1,...,w_n)$. One popular model assumes
that $(\log S_1,...,\log S_n)$ is a multivariate Gaussian random
vector. The value of the portfolio, $S= \sum_{i=1}^n w_i S_i$, is
then the sum of correlated log-normal random variables (see
\cite{DufPan97,Ahs78,BasSha01,GHS00,Due04}). Without loss of generality, we let $\sum w_{i}=n$.

One typical situation is that the portfolio size is large and the asset prices are usually highly correlated. One may employ a latent space approach used in the literature of social network. More specifically, we construct a Gaussian process $\{f(t): t\in T\}$ and associate each asset $i$ with a latent variable $t_i \in T$ so that $\log S_i = f(t_i)$. Then, the log asset prices fall into a subset of the continuous Gaussian process.
Further, there exists a (deterministic) process $w(t)$ so that $w(t_{i})= w_{i}$. Then, the total asset one unit share value of the portfolio is $\frac 1 n \sum w_{i} S_{i} = \frac 1 n \sum w(t_{i}) e^{f(t_{i})}$.

In this representation, the dependence among two assets is determined by $|t_i-t_j|$, which indicates the  economical distance between two firms. For instance, if firm $i$ is the supplier of firm $j$ then $t_i - t_j$ tends to be small. The spatial index $t$ may also include other social-economical indices.
This latent space approach has become popular in recent social network studies. For instance,
\cite{HRH02} considers a graph of $n$ nodes. The authors associate each node $i$ with a spatial latent variable $t_i$ and  model the probability of generating an edge between two nodes ($i$ and $j$) in a graph as a logistic function of $|t_i - t_j|$. Similar latent space models that project nodes onto a latent space can be found in \cite{HRT}.
Other approaches to represent interactions among variables via latent structures have been used;  see, for instance, \cite{XFS10,Sni02} and references therein.

In the asymptotic regime that $n\rightarrow \infty$ and the correlations among the asset prices become close to one, the subset $\{t_{i}\}$ becomes denser in $T$. Ultimately, we obtain the limit
$$\frac 1 n \sum_{i=1}^{n} w_{i} S_{i} \rightarrow \int w(t) e^{f(t)} h(t) dt$$
where $h(t)$ indicates the limiting spatial distribution of $\{t_{i}\}$ in $T$. Let $\mu(t)= \log w(t) + \log h(t)$. Then the tail probability of the (limiting) unit share price is
$$P\left(\int e^{f(t) + \mu(t)}dt > b\right).$$
It is necessary to include a varying mean in the above representation to incorporate the variation of the weights assigned to different assets and the inhomogeneity of the limiting distribution of $t_{i}$'s.

%then as the portfolio size tends to infinity
%and the prices become more correlated, the limit of an unit share
%price of the portfolio is $\lim_{n\rightarrow \infty} S/n = \int_T
%e^{f(t)} \nu(dt) = \int_T e^{f(t)+\mu(t)} dt$, where $\nu$ is a
%positive measure absolutely continuous with respect to Lebesgue
%measure with density $e^{\mu(t)}$.

\section{Main results}\label{SecMain}

\subsection{Problem setting}

Consider a homogeneous Gaussian random field $\{f(t): t\in T\}$
living on a domain $T\subset R^d$. Let the covariance function be
$$C(t-s) = Cov(f(t),f(s)).$$
We impose the following assumptions:
\begin{itemize}
  \item[C1] $f$ is homogenous with $Ef(t)=0$ and $Ef^2(t) =1$.
  \item[C2] $f$ is almost surely at least three times differentiable with respect to $t$ and $\mu(t)\in C^3(T)$.
  \item[C3] $T$ is a $d$-dimensional Borel measurable compact set of $R^d$ with piecewise smooth boundary.
  \item[C4] The Hessian matrix of $C(t)$ at the origin is $-I$, where $I$ is the $d\times d$ identity matrix.
  \item[C5] For each $t\in R^d$, the function $C(\lambda t)$ is a non-increasing function of $\lambda \in R^+$.
  \item[C6] If $\mu(t)$ is not a constant, the maximum of $\mu(t)$ is not attained at the boundary of $T$.
\end{itemize}
Let $N(\cdot)$ be a point process such that conditional on $\{f(t):
t\in T\}$, $N(\cdot)$ is distributed as a Poisson process with
intensity $\lambda (t) = e^{\mu(t) + \sigma f(t)}$. For each Borel measurable set $A\subset T$, let
\begin{equation}\label{IntI}
  I(A) = \int_A e^{\mu(t) + \sigma f(t)}dt.
\end{equation}
 Throughout this paper, we are interested in developing closed form approximations to
\begin{equation}\label{Tails}
  P\left(\int_T e^{\mu(t)+\sigma f(t)}dt > b\right), \quad \mbox{and }~~P(N(T)>b)
\end{equation}
%and
%\begin{equation}\label{TailCounts}
%P(N(A)>b)
%\end{equation}
as $b\rightarrow \infty$.

\begin{remark}
Condition C1 assumes unit variance. We treat the standard deviation $\sigma$ as an additional parameter and consider $\int e^{\mu(t) + \sigma f(t)}dt$. Condition C2 is rather a strong assumption. It implies that $C(t)$ is at least 6 times differentiable and the first, third, and fifth derivatives at the origin are all zero. Condition C3 restricts the results to finite horizon. Condition C4 is introduced to simplify notations. For any Gaussian process $g(t)$ with covariance function $C_g(t)$ and $\Delta C_g(0) = -\Sigma$ and $\det (\Sigma)>0$, C4 can be obtained by an affine transformation by letting $g(t)= f(\Sigma^{1/2}t)$ and
$$\int_T e^{\mu(t) + \sigma g(t)}dt = \det(\Sigma^{-1/2})\int _{\{s: \Sigma^{-1/2}s\in T\}}e^{\mu(\Sigma^{-1/2}s)+\sigma f(s)}ds,$$
where for each positive semi-definite matrix $\Sigma$ we let
$\Sigma^{1/2}$ be a symmetric matrix such that
$\Sigma^{1/2}\Sigma^{1/2}=\Sigma$. Conditions C5 and C6 are
imposed for technical reasons.

%\jc{More generally, one my consider a smooth function $\Sigma(t): T \rightarrow \mathcal P$, where $\mathcal P$ is set of $d\times d$ positive definitive matrices. Let $g(t) = f(\Sigma(t) t)$. Note that $g(t)$ is no longer stationary and is called locally stationary. With the locally stationary structure, we have the following identity.
%$$\int_T e^{\mu(t) + \sigma g(t)}dt = \int _{\{s: \Sigma^{-1/2}(s)s\in T\}}\det(\Sigma(s))^{-1/2}e^{\mu(\Sigma^{-1/2}(s)s)+\sigma f(s)}ds.$$}
\end{remark}

\begin{remark}
The setting in \eqref{Tails} also incorporates the case in which the
integral is with respect to other measures with smooth densities
with respect to the Lebesgue measure. Then, if $\nu(dt) = \kappa(t)
dt$, we will have that
$$\int_A e^{\mu(t) +\sigma f(t) }\nu(dt) = \int_A e^{\mu(t)+ \log \kappa (t) +\sigma f(t)} dt,$$
which shows that the density can be absorbed by the mean function as long as $\kappa (t)$ is bounded away from zero and infinity on $T$.
\end{remark}

\begin{remark}
The results presented in the current paper are directly applicable to some of the applications in Section \ref{SecApp} such as the approximation of $p$-value for simple null hypothesis. Some conditions required by the theorems may need to be relaxed to reflect practical circumstances for other applications. Nonetheless, the current analysis forms a standpoint of further study of more general cases.
\end{remark}

\subsection{Notations}

To simplify the discussion, we define a set of notations constantly used in the later development and provide some basic calculations of Gaussian random field. Let ``$\partial$'' denote the gradient and ``$\Delta$'' denote the Hessian matrix with respect to $t$. The notation ``$\partial^2$'' is used to denote the vector of second derivatives. The difference between $\partial^2 f(t)$ and $\Delta f(t)$ is that $\Delta f(t)$ is a $d\times d$ symmetric matrix whose diagonal and upper triangle consists of elements of $\partial^2 f(t)$.
Similarly, we will later use $\zz$ (and $\tilde{\zz}$) to denote the matrix version of the vector $z$ (and $\tilde{z}$), that is, $\mathbf{z}$ (and $\tilde{\mathbf{z}}$) is a symmetric matrix whose upper triangle consists of elements in the vector $z$ (and $\tilde{z}$).
Further, let $\partial_j f(t)$ be the partial derivative with respect to the $j$-th element of $t$. We define $u$ as the larger solution to
$$\left(\frac{2\pi}{\sigma}\right)^\frac{d}{2}u^{-\frac{d}{2}}e^{\sigma u}=b.$$ Note that when $b$ is large, the above equation generally has two solutions. One is on the order of $\log b$; the other one is close to zero. We choose the larger solution as our $u$.
Lastly, we define the following set of notations
\begin{eqnarray*}
\mu _{1}(t) &=&-(\partial _{1}C(t),...,\partial _{d}C(t)),\\\mu
_{2}(t) &=&\Big(\partial^2_{ii}C(t),i=1,...,d; \partial^2
_{ij}C(t),i=1,...,d-1,j=i+1,...,d\Big),\\ \mu_{02}^\top&=&\mu _{20}
=\mu _{2}(0).
\end{eqnarray*}

It is well known that (c.f. Chapter 5.5 of \cite{AdlTay07})
$\left(f(0),\partial^2 f(0), \partial f(0), f(t)\right)$ is a
multivariate Gaussian random vector with mean zero and covariance
matrix
\[
\left(
\begin{array}{cccc}
1 & \mu _{20} & 0 & C(t) \\
\mu _{02} & \mu _{22} & 0 & \mu _{2}^{\top }(t) \\
0 & 0 & I & \mu _{1}^{\top }(t) \\
C(t) & \mu _{2}(t) & \mu _{1}(t) & 1%
\end{array}%
\right)
\]
where the matrix $\mu_{22}$ is a $d(d+1)/2$-dimensional positive definite matrix and contains the 4th order spectral moments arranged in an appropriate order according to the order of elements in $\partial^2 f(0)$. Define
\begin{equation}\label{gamma}
\Gamma =\left(
\begin{array}{cc}
1 & \mu _{20} \\
\mu _{02} & \mu _{22}%
\end{array}%
\right). \end{equation}

For notation convenience, we write $a_u = O(b_u)$ if there exists a
constant $c>0$ independent of everything such that $a_u \leq cb_u$
for all $u>1$, and $a_u = o(b_u)$ if $a_u/b_u \rightarrow 0$ as
$u\rightarrow \infty$ and the convergence is uniform in other
quantities. We write $a_u = \Theta (b_u)$ if $ a_u =O(b_u)$ and $b_u
= O(a_u)$. In addition, we write $X_u = o_p(a_u)$ if $X_u/a_u \overset p
\rightarrow 0$ as $u\rightarrow \infty$ and $E^{\theta X_u/a_u}\rightarrow 1$ uniformly for $\theta$ over a compact interval around the origin. Similarly, we define that $X_u = O_p(a_u)$ if $E^{\theta X_u/a_u}$ is bounded away from zero and infinity for all $u\in R^+$ and $\theta$ in a compact interval around zero. We write $a_u\sim b_u$ if $a_u / b_u \rightarrow 1$ as $u\rightarrow \infty$.

\subsection{The main theorems}\label{SecThm}
The main theorems of this paper are presented as follows. The following theorem is the central result of this paper whose proof is provided in Section \ref{SecProof}.
\begin{theorem}\label{ThmG}
Consider a Gaussian random field $\{f(t): t\in T\}$ living on a
domain $T$ satisfying conditions C1-6. Let $I(T)$ be as defined in \eqref{IntI}. Then,
$$P(I(T) > b) \sim u^{d-1}\int_{T}\exp\left\{-\frac{(u-\mu_\sigma(t))^2}{2}\right\}\cdot H(\mu,\sigma,t)dt,$$
as $b\rightarrow \infty$, where
\begin{equation}\label{musig}
\mu_{\sigma}(t)= \mu(t)/\sigma,
\end{equation}
$u$ is the larger solution to
$$\left(\frac{2\pi}{\sigma}\right)^\frac{d}{2}u^{-\frac{d}{2}}e^{\sigma u}=b,$$
$H(\mu,\sigma,t)$ is defined as
\begin{eqnarray}
&&\frac{|\Gamma|^{-\frac{1}{2}}}{(2\pi)^\frac{(d+1)(d+2)}{4}}
\exp\left\{\frac{\mathbf{1}^T\mu_{22}\mathbf{1}+\sum_i\partial_{iiii}^{4}C(0)}{8\sigma^2}
+\frac{d\cdot \mu_\sigma(t)+Tr(\Delta
\mu_\sigma(t))}{2\sigma}+|\partial\mu_\sigma(t)|^2\right\}.\nonumber\nonumber\\
&& \times \int_{z\in R^{d(d+1)/2}}\exp\left\{ - \frac{1}{2}
\left[\frac{|\mu_{20}\mu_{22}^{-1}z|^2}{1-\mu_{20}\mu_{22}^{-1}\mu_{02}}
+\left
|\mu_{22}^{-1/2}z-\frac{\mu_{22}^{1/2}\mathbf{1}}{2\sigma}\right|^2
\right]\right\}dz,\nonumber\end{eqnarray}
and
$$\mathbf{1} =\mathbf{(}\underset{d}{\underbrace{1,...,1}},\underset{d(d-1)/2%
}{\underbrace{0,...,0}}\mathbf{)}^{\top }.$$
\end{theorem}

\begin{corollary}\label{CorGRF}
Under the conditions of Theorem \ref{ThmG}, if $\mu(t)$ has one
unique maximum in $T$ denoted by $t_*$, then
\begin{eqnarray*}
P(I(T) > b)   \sim (2\pi )^{d/2}\det (\Delta \mu_\sigma (t_{\ast
}))^{-1/2}H(\mu,\sigma,t_{\ast
})u^{d/2-1}\exp \left\{ -\frac{(u-\mu_{\sigma} (t_{\ast }) )^{2}}{2}%
\right\} .
\end{eqnarray*}
\end{corollary}

\begin{proof}
[Proof of Corollary \ref{CorGRF}] The result is immediate by
expanding $\mu_{\sigma}(t)$ around $t_*$ up to the second order.
\end{proof}

\begin{theorem}\label{ThmC}
Assume that the Gaussian process $f(t)$ satisfies the conditions in
Theorem \ref{ThmG}. Consider a point process $\{N(A): A\in \mathcal
B(T)\}$, where $\mathcal B(T)$ denotes the Borel subsets of $T$.
Suppose that there exists a process $\log \lambda(t) = \mu(t) +
\sigma f(t)$ such that given $\{\lambda(t): t\in T\}$, $N(\cdot)$ is
a Poisson process with intensity $\lambda (t)$. Then,
$$P(N(T)>b)\sim  P(I(T)>b)$$
as $b\rightarrow \infty$.
\end{theorem}

\begin{proof}[Proof of Theorem \ref{ThmC}]
We prove this approximation from both sides. For $\varepsilon >0$ small
enough, we have that%
\begin{eqnarray}\label{tt}
P\left( N(T)>b\right)  &\geq &P\left(N(T)>b~;~ I(T)\geq b+b^{1/2+\varepsilon }\right) \\
&=&(1+o(1))P\left(I(T)\geq b+b^{1/2+\varepsilon }\right) \notag\\
&=&(1+o(1))P\left(I(T)\geq b\right).\notag
\end{eqnarray}%
The second step is due to the fact that conditional on $I(T)$
$$\frac{N(T) - I(T)}{\sqrt{I(T)}} \rightarrow N(0,1) $$
in distribution as $I(T)\rightarrow \infty$. Therefore, we obtain that
\begin{equation*}
P\left(N(T)>b~|~I(T)\geq b+b^{1/2+\varepsilon }\right)\rightarrow 1.
\end{equation*}
Together with the fact that
$$P\left(N(T)>b~;~ I(T)\geq b+b^{1/2+\varepsilon }\right) = P\left(N(T)>b~|~I(T)\geq b+b^{1/2+\varepsilon }\right)P\left(I(T)\geq b+b^{1/2+\varepsilon }\right),$$
we obtain the second step of \eqref{tt}.
The last step that $P(I(T)\geq b+b^{1/2+\varepsilon })=(1+o(1))P(I(T)\geq b)$
is a direct application of Theorem \ref{ThmG}. For the upper bound, we have that%
\begin{eqnarray*}
P\left( N(T)>b\right)  &=&P\left(N(T)>b~;~I(T)\geq
b-b^{1-\varepsilon
}\right)+P\left(N(T)>b~;~I(T)\leq b-b^{1-\varepsilon }\right) \\
&\leq &(1+o(1))P\left(I(T)>b\right)+P\left(N(T)>b~|~I(T)=b-b^{1-\varepsilon }\right) \\
&=&(1+o(1))P(I(T)>b).
\end{eqnarray*}%
The last step uses the fact that%
\begin{equation*}
P\left(N(T)>b~|~I(T)=b-b^{1-\varepsilon }\right)\leq \exp \left\{
-(1/2+o(1))b^{1-2\varepsilon }\right\} =o(1)P(I(T)>b).
\end{equation*}%
The bound of the tail of a Poisson distribution can be derived by the
standard technique of large deviations theory \cite{Dembo} and
therefore is omitted.
\end{proof}

\begin{remark}
The result in Theorem \ref{ThmC} suggests that an observation of a
large number of points in a region $T$ is mainly caused by a high
level of its underlying intensity. Technically, this is because the
distribution of $N(T)$ can be roughly considered as a convolution of
the distribution of $\int e^{f(t)}dt$ and a Poisson distribution.
Note that $\int e^{f(t)}dt$ is approximately a log-normal random
variable, which has a much heavier tail than that of a Poisson
random variable. Therefore, the tail behavior of $N(T)$ is mostly
dominated by the tail of its underlying intensity.
\end{remark}

\subsection{The change of measure}\label{SecChange}

In this subsection, we propose a change of measure $Q$ which is
central to the proof of Theorem \ref{ThmG}. Let $P$ be the original measure. The
measure $Q$ is defined such that $P$ and $Q$ are mutually absolutely
continuous with the Radon-Nikodym derivative being
\begin{equation}\label{RN}\frac{dQ}{dP} = \int_T\frac{1}{mes(T)}\cdot \frac{\exp\left\{-\frac{1}{2}(f(t)-u+\mu_\sigma(t))^2\right\}} {\exp\left\{-\frac{1}{2}f(t)^2\right\}}dt,\end{equation}
where $mes(\cdot)$ denotes Lebesgue measure. This change of measure is first proposed by \cite{NSY08} to derive the high excursion probabilities of approximately Gaussian processes.
It is more intuitive to describe the measure $Q$ from a simulation
point of view. In order to simulate $f(t)$ under the measure $Q$,
one can do the following two steps:
\begin{enumerate}
  \item Simulate a random variable $\tau$ uniformly over $T$ with respect to the Lebesgue measure.
  \item Given the realized $\tau$, simulate the Gaussian process $f(t)$ with mean $(u-\mu_\sigma(\tau)) C(t-\tau)$ and covariance function $C(t)$.

%  \item Given the realized $\tau$, simulate $f(\tau)$ from $N(u-\mu_\sigma(\tau),1)$.
%  \item Given $(\tau,f(\tau))$, simulate $\{f(t): t\neq \tau\}$ from the original conditional distribution under $P$.
\end{enumerate}
It is not hard to verify that the above two-step procedure is
consistent with the Radon-Nikodym derivative in \eqref{RN}. The
measure $Q$ is designed such that the distribution of $f$ under the
measure $Q$ is approximately the conditional distribution of $f$
given $\int_T e^{f(t)}dt >b $ under the measure $P$. Under $Q$, a
random variable $\tau$ is first sampled uniformly over $T$, then
$f(\tau)$ is simulated with a large mean at level $u-\mu_\sigma(\tau)$. This implies
that the high level of the integral $\int_T e^{\mu(t) + \sigma
f(t)}dt$ is mostly caused by the fact that the field reaches a high
level at one location $t^*$ and such a location $t^*$ is very close
to $\tau$. Therefore, the random index $\tau$ localizes the maximum
of the field. In particular, one can write the tail probability as
\begin{eqnarray*}
P\left( \int_{T}e^{\mu (t)+\sigma f(t)}dt>b\right) &=&E^{Q}\left[ \frac{dP}{dQ}%
;\int_{T}e^{\mu (t)+\sigma f(t)}dt>b\right],
%&=&E^{Q}\left\{ E\left[ \left.
%\frac{dP}{dQ};\int_{T}e^{\mu (t)+\sigma f(t)}dt>b\right\vert \tau ,f(\tau )%
%\right] \right\} ,
\end{eqnarray*}
where we use $E^Q$ to denote the expectation under $Q$ and $E$ to denote that under $P$.
%The last step of the above display is due to the fact that given %$(\tau,f(\tau))$ the conditional distributions of $f$ under $P$ and $Q$ are %the same.

In what follows, we explain the main result in Theorem \ref{ThmG}
and how the change of measure helps in deriving the asymptotics. To
simplify the discussion, we proceed by assuming that $\mu(t) \equiv
0$ and $\sigma=1$. Upon considering zero (constant) mean, we obtain
from the result of Theorem \ref{ThmG} that $P(\int_T e^{f(t)}dt >
b)= \Theta(1) P(\sup_T f(t) > u)$ (c.f. \cite{AdlTay07}). This
suggests that the large value of the exponential integral at the
level $b$ is largely caused by the high excursion of $\sup_T f(t)$
at a level $u$. The conditional distribution of $f(t)$ given a high
excursion at level $u$ (the Slepian model) is well known
(\cite{ATW09}). We proceed with a rough mean calculation. Suppose
that $f(t)$ attains a large value at the origin of level $u$. Then
the conditional field will have expectation $E[f(t)|f(0)=u] =
uC(t)$. We expand the covariance function as
$$uC(t)\approx u - \frac u 2 |t|^2.$$
Therefore, one may expect to choose $u$ such that conditional on $f(0) = u$
\begin{equation}\label{app}
\int _T e^{f(t)}dt \approx \int_{R^d} e^{u-\frac u 2 |t|^2}dt = \left(2\pi\right)^{d/2} u^{-d/2}e^u =b.
\end{equation}
This is precisely how $u$ is selected in Theorem \ref{ThmG}. The above calculation ignores the higher order expansions of $C(t)$ and the deviation of the conditional field from its expectation. It turns out that these variations do not affect the asymptotic decaying rate of the tail probability. They only contribute to the constant term.

\section{Proof of Theorem \ref{ThmG}}\label{SecProof}

%In this section, we provide the proof of Theorem \ref{ThmG}. Note that
%$$P\left( \int_{T}e^{\mu(t) + \sigma f(t)}dt>b\right) = \frac 1 {mes (T)}\int_T E\left[ \left. \frac{dP}{dQ};\int_{T}e^{\mu(t) + \sigma f(t)}dt>b\right\vert
%t ,f(t )\right] dt.$$
%In what follows, look closer at the integrant.

%We first describe several notations.
%For the second derivatives of $f$, we
%use $\Delta f(\tau )$ to denote the Hessian matrix and $\partial ^{2}f(\tau
%) $ to denote the vector of second derivatives. Therefore, $\Delta f(\tau )$
%is a symmetric matrix whose upper triangle is filled with the elements of $%
%\partial ^{2}f(\tau )$.

The proof of Theorem \ref{ThmG} requires several lemmas. To
facilitate the reading, we arrange their statements in
Section \ref{SecLem}.

Note that%
\begin{eqnarray*}
P\left( \int_{T}e^{\mu (t)+\sigma f(t)}dt>b\right) &=&E^{Q}\left[ \frac{dP}{dQ}%
;\int_{T}e^{\mu (t)+\sigma f(t)}dt>b\right] \\
&=&\int_{T}\frac{1}{mes(T)}E^{Q}%
\left[ \left. \frac{dP}{dQ};\int_{T}e^{\mu (t)+\sigma f(t)}dt>b\right\vert
\tau \right] d\tau .
\end{eqnarray*}
Furthermore, we use the notation that $E_{\tau }^{Q}[\cdot
]=E^{Q}[\cdot |\tau ]$. For each $\tau $, we plug in \eqref{RN} and further write the
expectation inside the above integral as
\begin{eqnarray} \label{obj}
&&~~~~~E_{\tau }^{Q}\left[ \frac{dP}{dQ};\int_{T}e^{\mu (t)+\sigma f(t)}dt>b\right]\\
&&~~~~~~=mes(T)E_{\tau }^{Q}\left[ \frac{1}{\int_{T}e^{-\frac{1}{2}\left(
f(t)-u+\mu _{\sigma }(t)\right) ^{2}+\frac{1}{2}f^{2}(t)}dt};\int_{T}e^{\mu
(t)+\sigma f(t)}dt>b\right]  \notag\\
&&~~~~~~=mes(T)e^{u^{2}/2}E_{\tau }^{Q}\left[ \frac{1}{\int_{T}e^{(u-\mu _{\sigma
}(t))(f(t)+\mu _{\sigma }(t))+\frac{1}{2}\mu _{\sigma }^{2}(t)}dt}%
;\int_{T}e^{\mu (t)+\sigma f(t)}dt>b\right] .  \notag
\end{eqnarray}
We write
\begin{equation}\label{lambda}
\Lambda(\tau)=e^{u^{2}/2}E_{\tau }^{Q}\left[ \frac{1}{\int_T
e^{(u-\mu _{\sigma
}(t))(f(t)+\mu _{\sigma }(t))+\frac{1}{2}\mu _{\sigma }^{2}(t)}dt}%
;\int_{T}e^{\mu (t)+\sigma f(t)}dt>b\right].
\end{equation}
Note that conditional on $\tau $, for every set $A$,
\begin{equation}\label{equav}
Q\left( f(\cdot )\in A|\tau \right) =P\Big(f(\cdot )+(u-\mu _{\sigma }(\tau
))C(\cdot -\tau )\in A\Big),
\end{equation}%
that is, the conditional distribution of $f(t)$ (given $\tau$) under $Q$  equals to the distribution of $%
f(t)+(u-\mu _{\sigma }(\tau ))C(t-\tau )$ under $P$. This equivalence can be derived from the  two-step simulation procedure in Section \ref{SecChange}. Therefore, we can simply replace $f(t)$ by $f(t)+(u-\mu _{\sigma }(\tau ))C(t-\tau )$, replace $Q$ by $P$, and write
\begin{eqnarray}\label{Integral}
\Lambda(\tau) &=& e^{u^{2}/2}E\biggr[ \frac{1}{\int_{T}e^{(u-\mu
_{\sigma }(t))\left[ f(t)+(u-\mu _{\sigma }(\tau ))C(t-\tau )+\mu
_{\sigma }(t)\right] +\frac{1}{2}\mu _{\sigma
}^{2}(t)}dt};\notag\\
&&~~~~~~~~~~\int_{T}e^{\sigma \left\{ f(t)+(u-\mu _{\sigma }(\tau
))C(t-\tau )+\mu _{\sigma }(t)\right\} }dt>b\biggr]
\end{eqnarray}%
Let
\begin{eqnarray}
\mathcal{E}_{b} &=&\left\{ \int_{T}e^{\sigma \left\{ f(t)+(u-\mu _{\sigma
}(t))C(t-\tau )+\mu _{\sigma }(t)\right\} }dt>b\right\} ,  \label{Event} \\
K &=&\int_{T}e^{(u-\mu _{\sigma }(t))\left[ f(t)+(u-\mu _{\sigma }(\tau
))C(t-\tau )+\mu _{\sigma }(t)\right] +\frac{1}{2}\mu _{\sigma }^{2}(t)}dt.
\label{LRI}
\end{eqnarray}%
Then, (\ref{Integral}) can be written as%
\begin{equation}
\Lambda (\tau)=e^{u^{2}/2}\int E\left[ \left. K^{-1};\mathcal{E}%
_{b}\right\vert f(\tau )=w,\partial f(\tau )=\tilde{y},\partial ^{2}f(\tau )=%
\tilde{z}\right] \times
h(w,\tilde{y},\tilde{z})dwd\tilde{y}d\tilde{z}, \label{Int}
\end{equation}%
where $h(w,\tilde{y},\tilde{z})$ is the density function of
$\left(f(\tau
),\partial f(\tau ),\partial ^{2}f(\tau)\right)$ evaluated at $(w,\tilde{y},%
\tilde{z})$.

For a given $\delta'>0$ small enough, we consider two cases for $\tau$:  first, $\{t:|t-\tau |\leq u^{-1/2+\delta ^{\prime }}\}\subset T$ and
otherwise. For the first situation, $\tau$ is ``far away'' from the
boundary of $T$, which is the important case in our analysis. For the second
situation, $\tau$ is close to the boundary. We will show that the
second situation is of less importance given that the maximum of
$\mu(t)$ is attained at the interior of $T$.

For the first situation, the analysis consists of three main parts.

\begin{description}
\item[Part 1] Conditional on $\left(\tau ,f(\tau ),\partial f(\tau ),\partial
^{2}f(\tau )\right)$, we study the event%
\begin{equation}
\mathcal{E}_{b}=\left\{ \int_{T}e^{\sigma \left\{ f(t)+(u-\mu _{\sigma
}(\tau ))C(t-\tau )+\mu _{\sigma }(t)\right\} }dt>b\right\} ,  \label{cond}
\end{equation}%
and write the occurrence of this event almost as a deterministic function of $%
f(\tau )$, $\partial f(\tau )$, and $\partial ^{2}f(\tau )$.

\item[Part 2] Conditional on $\left(\tau ,f(\tau ),\partial f(\tau ),\partial
^{2}f(\tau )\right)$, we write $K$ defined in (\ref{LRI}) as a
function of $f(\tau )$, $\partial f(\tau )$,  $\partial ^{2}f(\tau
)$ and a small correction term.

\item[Part 3] We combine the results from the first two parts and obtain an
approximation of (\ref{obj}) through the right-hand-side of (\ref{Int}).
\end{description}

All the subsequent derivations are conditional on a specific value of
$\tau$.

\subsection*{Preliminary calculations}

For $0<\varepsilon<\delta' $ sufficiently small, let%
\begin{equation}
\mathcal{L}_{Q}=\left\{\left\vert f(\tau )-u+\mu _{\sigma }(\tau
)\right\vert \leq u^{1/2+\varepsilon },|\partial f(\tau
)|<u^{1/2+\varepsilon },|\partial ^{2}f(\tau )-(u-\mu _{\sigma
}(\tau ))\mu _{02}|<u^{1/2+\varepsilon }\right\}. \label{LQ}
\end{equation}%
According to Lemma \ref{LemLocal}, we only need to consider the integral on the
set $\mathcal{L}_Q$, that is,%
\begin{eqnarray*}
&&E_{\tau }^{Q}\left[ \frac{1}{\int \exp \left\{(u-\mu _{\sigma
}(t))(f(t)+\mu _{\sigma }(t))+\frac{1}{2}\mu _{\sigma }^{2}(t)\right\} dt}%
;\int_{T}e^{\mu (t)+\sigma f(t)}dt>b,\mathcal{L}_{Q}\right] .\end{eqnarray*}
The above display equals to
$$E\left[ K^{-1};\mathcal{E}_{b},\mathcal{L}\right] ,$$
where
\begin{equation}
\mathcal{L}=\left\{|f(\tau )|\leq u^{1/2+\varepsilon },|\partial
f(\tau )|<u^{1/2+\varepsilon },|\partial ^{2}f(\tau
)|<u^{1/2+\varepsilon }\right\}, \label{local}
\end{equation}
corresponds to $\mathcal L_Q$ under the transform \eqref{equav}.
Therefore, throughout the rest of the proof, all the derivations are
on the set $\mathcal{L}$. Note that the sets $\mathcal{L}_{Q}$ and
$\mathcal{L}$ depend on $\tau $ and $u$. Since all the subsequent
derivations are for specific $\tau $ and $u$, we omit the indices of
$\tau $ and $u$ in the notations $\mathcal{L}$ and
$\mathcal{L}_{Q}$.

We first provide the Taylor expansions for $f(t)$, $C(t)$, and $\mu(t)$.
\begin{itemize}
\item Expansion of $f(t)$ given $\left(f(\tau
),\partial f(\tau ),\partial ^{2}f(\tau )\right)$. Let $t-\tau
=((t-\tau )_{1},...,(t-\tau )_{d})$. Conditional on $\left(f(\tau
),\partial f(\tau ),\partial ^{2}f(\tau )\right)$, we first expand
the random
function%
\begin{eqnarray}
f(t) &=&E\left[f(t)|f(\tau ),\partial f(\tau ),\partial ^{2}f(\tau
)\right]+g(t-\tau)
\label{expf} \\
&=&f(\tau )+\partial f(\tau )^{\top }(t-\tau )+\frac{1}{2}(t-\tau )^{\top
}\Delta f(\tau )(t-\tau )  \notag \\
&&+g_{3}(t-\tau )+R_{f}(t-\tau )+g(t-\tau ),  \notag
\end{eqnarray}%
where%
\begin{equation*}
g_{3}(t-\tau )=\frac{1}{6}\sum_{i,j,k}E\left[\partial
_{ijk}^{3}f(\tau )|f(\tau ),\partial f(\tau ),\partial ^{2} f(\tau
)\right](t-\tau )_{i}(t-\tau )_{j}(t-\tau )_{k}.
\end{equation*}%
Note that $\partial _{ijk}^{3}f(\tau )$ is independent of $(f(\tau ),\Delta
f(\tau ))$ and
\begin{equation*}
E\left[\partial _{ijk}^{3}f(\tau )|f(\tau ),\partial f(\tau
),\partial ^{2} f(\tau )\right]=-\sum_{l}\partial^{4}
_{ijkl}C(0)\partial _{l}f(\tau ).
\end{equation*}%
$g(t)$ is a mean zero Gaussian random field such that
$Eg^{2}(t)=O(|t|^{6})$ as $t\rightarrow 0$. In addition, the
distribution of $g(t)$ is independent of $\tau, f(\tau), \partial
f(\tau)$, and $\partial^2 f(\tau)$. $R_{f}(t-\tau )=O(|t-\tau
|^{4})$ is the remainder term of the Taylor expansion of
$E\left[f(t)|f(\tau ),\partial f(\tau ),\partial ^{2}f(\tau
)\right]$.

\item Expansion of $C(t)$:
\begin{equation}
C(t)=1-\frac{1}{2}t^{\top }t+C_{4}(t)+R_{C}(t),  \label{expc}
\end{equation}%
where $R_{C}(t)=O(|t|^{6})$ and%
\begin{equation*}
C_{4}(t)=\frac{1}{24}\sum_{ijkl}\partial _{ijkl}^{4}C(0)t_{i}t_{j}t_{k}t_{l}.
\end{equation*}%
\item Expansion of $\mu(t)$:
%Lastly, we have the expansion for $\mu _{\sigma }(t)$%
\begin{equation}
\mu _{\sigma }(t)=\mu _{\sigma }(\tau )+\partial \mu _{\sigma }(\tau )^{\top
}(t-\tau )+\frac{1}{2}(t-\tau )^{\top }\Delta \mu _{\sigma }(\tau )(t-\tau
)+R_{\mu }(t-\tau ),  \label{expmu}
\end{equation}%
where $R_{\mu }(t-\tau )=O(|t-\tau |^{3})$.
\end{itemize}

 Let $I$ be the $d\times d$
identity matrix. We define the following notations that will be constantly
used later,%
\begin{eqnarray*}
\tilde{u} &=&u-\mu _{\sigma }(\tau ),\qquad \tilde{y}=\partial f(\tau ),\qquad \tilde{\zz}%
=\Delta f(\tau ), \\
y &=&\partial f(\tau )+\partial \mu _{\sigma }(\tau ),\qquad \zz=\Delta f(\tau )+\mu
_{\sigma }(\tau )I+\Delta \mu _{\sigma }(\tau ), \\
R(t) &=&R_{f}(t)+(u-\mu _{\sigma }(\tau ))R_{C}(t)+R_{\mu }(t).
\end{eqnarray*}
As mentioned earlier, we let $z$ and $\tilde z$ be the vector version of the matrices $\zz$ and $\tilde {\zz}$.

Now, we start to carry out our three-step program.

\subsection*{Part 1}

All the derivations in this part are conditional on specific values
of $\tau$, $f(\tau )$, $\partial f(\tau )$, and $\partial ^{2}f(\tau
)$. Define
\begin{equation*}
I_{1}\triangleq \int_{T}e^{\sigma \left\{ f(t)+(u-\mu _{\sigma }(\tau
))C(t-\tau )+\mu _{\sigma }(t)\right\} }dt.
\end{equation*}%
We insert the expansions in (\ref{expf}), (\ref{expc}) and (\ref{expmu})
into the expression of $I_{1}$ and obtain that%
\begin{eqnarray}
I_{1} &=&\int_{t\in T}\exp \biggr\{\sigma \biggr[ f(\tau )+\partial
f(\tau )^{\top }(t-\tau )+\frac{1}{2}(t-\tau )^{\top }\Delta f(\tau
)(t-\tau )+g_{3}(t-\tau
)   \notag\\
&&+R_{f}(t-\tau )+g(t-\tau )  \notag \\
&&+\left(u-\mu _{\sigma }(\tau )\right)\left(1-\frac{1}{2}(t-\tau
)^{\top }(t-\tau
)+C_{4}(t-\tau )+R_{C}(t-\tau )\right)  \notag \\
&&+\mu _{\sigma }(\tau )+\partial \mu _{\sigma }(\tau )^{\top }(t-\tau )+%
\frac{1}{2}(t-\tau )^{\top }\Delta \mu _{\sigma }(\tau )(t-\tau
)+R_{\mu }(t-\tau )\biggr]\biggr\}dt.  \label{I1}
\end{eqnarray}%
We write the exponent inside the integral in a quadratic form of $(t-\tau)$ and obtain that%
\begin{eqnarray}
I_{1} &=&\exp \left\{ \sigma u+\sigma f(\tau )+\frac{\sigma }{2}y^{\top }(uI-%
\zz)^{-1}y\right\}   \notag \\
&&\int_{t\in T}\exp \left\{ -\frac{\sigma }{2}(t-(uI-\zz)^{-1}y)^{\top }(uI-%
\zz)\left( s-(uI-\zz)^{-1}y\right) \right\}   \notag \\
&&\times \exp \left\{ \sigma g_{3}\left( t\right) +\sigma (u-\mu _{\sigma
}(\tau ))C_{4}\left( t\right) +\sigma R\left( t\right) \right\} \times \exp
\left\{ \sigma g\left( t\right) \right\} dt  \label{l4}
\end{eqnarray}%
Further, consider the change of variable that $s=(uI-\zz)^{1/2}(t-\tau )$,
write the big integral in above display as a product of expectations and a
normalizing constant, and obtain that%
\begin{eqnarray}
I_{1} &=&\det (uI-\zz)^{-1/2}\exp \left\{ \sigma u+\sigma f(\tau )+\frac{%
\sigma }{2}y^{\top }(uI-\zz)^{-1}y\right\}   \notag \\
&&\times \int_{(uI-\mathbf{z})^{-\frac{1}{2}}s+\tau \in T}\exp \left\{ -%
\frac{\sigma }{2}\left( s-(uI-\zz)^{-1/2}y\right) ^{\top }\left( s-(uI-\zz%
)^{-1/2}y\right) \right\} ds  \notag \\
&&\times E\left[ \exp \left\{ \sigma g_{3}\left( (uI-\zz)^{-\frac{1}{2}%
}S\right) +\sigma (u-\mu _{\sigma }(\tau ))C_{4}\left( (uI-\zz)^{-\frac{1}{2}%
}S\right) +\sigma R\left( (uI-\zz)^{-\frac{1}{2}}S\right) \right\} \right]
\notag \\
&&\times E\left[ \exp \left\{ \sigma g\left( (uI-\zz)^{-\frac{1}{2}}\tilde{S}%
\right) \right\} \right] .  \notag
\end{eqnarray}
The two expectations in the above display are taken with respect to $S$ and
$\tilde{S}$ given the process $g(t)$. $S$ is a random variable
taking values in the set $\left\{s:(uI-\zz)^{-1/2}s+\tau \in
T\right\}$ with density proportional to
\begin{equation*}
\exp \left\{ -\frac{\sigma }{2}\left(s-(uI-\zz)^{-1/2}y\right)^{\top }\left(s-(uI-\zz)^{-1/2}y\right)%
\right\}
\end{equation*}%
and $\tilde{S}$ is a random variable taking values in the set $\left\{s:(uI-\zz)^{-1/2}s+\tau \in T\right\}$ with density proportional to%
\begin{eqnarray}
&&\exp \left\{-\frac{\sigma
}{2}\left(s-(uI-\zz)^{-1/2}y\right)^{\top
}\left(s-(uI-\zz)^{-1/2}y\right)\right\}
\notag\\
&&~~~~~\times\exp\left\{\sigma g_{3}\left((uI-\zz)^{-\frac{1}{2}}s\right)+\sigma (u-\mu
_{\sigma }(\tau ))C_{4}\left((uI-\zz)^{-\frac{1}{2}}s\right)+\sigma
R\left((uI-\zz)^{-\frac{1}{2}}s\right)\right\}.  \notag\\
\label{den}
\end{eqnarray}%
Together with the definition of $u$ that
\begin{equation*}
\left( \frac{2\pi }{\sigma }\right) ^{d/2}u^{-d/2}e^{\sigma u}=b,
\end{equation*}%
we obtain that
\begin{equation*}
I_1=\int_{T}e^{\sigma \left\{ f(t)+(u-\mu _{\sigma }(t))C(t-\tau )+\mu _{\sigma
}(t)\right\} }dt>b
\end{equation*}%
if and only if%
\begin{eqnarray}
I_{1} &=&\det (uI-\zz)^{-1/2}\exp \left\{ \sigma u+\sigma f(\tau )+\frac{%
\sigma }{2}y^{\top }(uI-\zz)^{-1}y\right\}    \notag\\
&&\times \int_{(uI-\mathbf{z})^{-\frac{1}{2}}s+\tau \in T}\exp \left\{ -%
\frac{\sigma }{2}\left( s-(uI-\zz)^{-1/2}y\right) ^{\top }\left( s-(uI-\zz%
)^{-1/2}y\right) \right\} ds  \notag\\
&&\times E\exp \left\{ \sigma g_{3}\left( (uI-\zz)^{-\frac{1}{2}}S\right)
+\sigma (u-\mu _{\sigma }(\tau ))C_{4}\left( (uI-\zz)^{-\frac{1}{2}}S\right)
+\sigma R\left( (uI-\zz)^{-\frac{1}{2}}S\right) \right\}   \notag \\
&&\times \exp \left\{ -u^{-1}\xi _{u}\right\}   \notag\\
&>&\left( \frac{2\pi }{\sigma }\right) ^{d/2}u^{-d/2}e^{\sigma u},\notag\\
\label{ineq}
\end{eqnarray}%
where
\begin{equation}\label{xi}
\xi _{u}=-u\log \left\{ E\exp \left[ \sigma g\left((uI-\zz)^{-\frac{1}{2}}\tilde{S}%
\right)\right] \right\}.
\end{equation}%
We take log on both sides and plug in the result of Lemma \ref{LemExp} that handles the big expectation term in \eqref{ineq}.
Then, the inequality \eqref{ineq} is equivalent to
\begin{eqnarray}
A \triangleq\sigma f(\tau )+\frac{\sigma }{2}y^{\top }(uI-\zz)^{-1}y-\frac{1}{2}\log
\det(I-u^{-1}\zz)+\sigma B+o(u^{-1})   >u^{-1}\xi _{u},  \notag\\
\label{A}
\end{eqnarray}%
where%
\begin{equation}\label{B}
B\triangleq-\frac{1}{8u}(u^{-1}Y+\mathbf{1}/\sigma )^{\top }\mu _{22}(u^{-1}Y+\mathbf{%
1}/\sigma )+\frac{\mathbf{1}^{\top }\mu _{22}\mathbf{1}}{8\sigma ^{2}u}+%
\frac{1}{8\sigma ^{2}u}\sum_{i}\partial _{iiii}^{4}C(0),
\end{equation}%
and
\begin{eqnarray*}
Y =\left\{y_{i}^{2},i=1,...,d;2y_{i}y_{j},1\leq i<j\leq d\right\}, \qquad
\mathbf{1} =\mathbf{(}\underset{d}{\underbrace{1,...,1}},\underset{d(d-1)/2%
}{\underbrace{0,...,0}}\mathbf{)}^{\top }.
\end{eqnarray*}%
Roughly speaking, according to Lemma \ref{LemRemainder}, the event $\mathcal{%
E}_{b}$ is the same as the event $\left\{A>O_{p}(u^{-3/2+3\delta
})\right\}$.

\subsection*{Part 2}

Similar to Part 1, all the derivations in this part are conditional
on $\left(\tau ,f(\tau ),\partial f(\tau ),\partial ^{2}f(\tau
)\right)$. We now proceed to the second part of the proof. More
precisely, we simplify the term $K$ defined as in (\ref{LRI}) and
write it as a deterministic function of $\left(f(\tau),\partial
f(\tau ),\partial ^{2}f(\tau )\right)$ with a small correction term.
For $\varepsilon<\delta <\delta'$ with all of them sufficiently
small, we let $\lambda _{u}=u^{-1/2+\delta }$. We first split the
integral into two parts, that is,
\begin{eqnarray*}
K &=&\int_{T}e^{(u-\mu _{\sigma }(t))\left[ f(t)+(u-\mu _{\sigma
}(\tau
))C(t-\tau )+\mu _{\sigma }(t)\right] +\frac{1}{2}\mu _{\sigma }^{2}(t)}dt \\
&=&\int_{|t-\tau |<\lambda _{u}}...+\int_{|t-\tau |>\lambda _{u}}... \\
&=&I_{2}+I_{3}.
\end{eqnarray*}%
For the leading term, note that $|t-\tau|\leq \lambda_n = u^{-1/2 + \delta}$. We insert the Taylor expansion of $\mu_\sigma(t)$
\begin{eqnarray*}
I_{2} &=&\int_{|t-\tau |<\lambda _{u}}e^{(u-\mu _{\sigma }(t))\left[
f(t)+(u-\mu _{\sigma }(\tau ))C(t-\tau )+\mu _{\sigma }(t)\right]
+\frac{1}{2}\mu _{\sigma }^{2}(t)}dt \\
&=&(1+o(1))e^{u^{2}-u\mu _{\sigma }(\tau )+\frac{1}{2}\mu^{2} _{\sigma }(\tau
)} \int_{|t-\tau |<\lambda _{u}}e^{(u-\mu _{\sigma }(t))\left[ f(t)+(u-\mu
_{\sigma }(\tau ))C(t-\tau )-u+\mu _{\sigma }(\tau )\right] }dt.
\end{eqnarray*}%
Let $\zeta _{u}=O(u^{-1/2+\delta })$. In what follows, we insert the
expansions in (\ref{expf}), (\ref{expc}), and (\ref{expmu}), write the exponent as a quadratic function of $t-\tau$, and obtain that on the set $\mathcal{L}$
\begin{eqnarray}
&&\int_{|t-\tau |<\lambda _{u}}e^{(u-\mu _{\sigma }(t))\left[
f(t)+(u-\mu _{\sigma }(\tau ))C(t-\tau )-u+\mu _{\sigma }(\tau
)\right] }dt
\notag\\
&=&\int_{|t-\tau |<\lambda _{u}}\exp \biggr\{(\tilde{u}+\zeta _{u})
\biggr[ f(\tau )+(t-\tau )^{\top }\tilde{y}-\frac{1}{2}(t-\tau
)^{\top }(\tilde{u}I-\tilde{\zz})(t-\tau )\notag \\
&&+g_{3}(t-\tau)+\tilde{u}C_{4}(t-\tau)+g(t-\tau)+O\left(u^{-3/2+3\delta }\right)\biggr] \biggr\}dt  \notag \\
&=&(1+o(1))\exp \left\{(\tilde{u}+\zeta _{u})\left(f(\tau
)+\frac{1}{2}\tilde{y}^{\top
}(\tilde{u}I-\tilde{\zz})^{-1}\tilde{y}\right)\right\}  \notag \\
&&\times \int_{|t-\tau |<\lambda _{u}}\exp \biggr\{(\tilde{u}+\zeta _{u})\biggr[-\frac{1%
}{2}\left(t-\tau -(\tilde{u}I-\tilde{\zz})^{-1}\tilde{y}\right)^{\top }\left(\tilde{u}%
I-\tilde{\zz}\right)\left(t-\tau -(\tilde{u}I-\tilde{\zz})^{-1}\tilde{y}\right)  \notag \\
&&+g_{3}(t-\tau)+\tilde{u}C_{4}(t-\tau)+g(t-\tau)+O\left(u^{-3/2+3\delta
}\right)\biggr]\biggr\}dt  \notag\\
\label{LRD}
\end{eqnarray}%
We consider the change of variable that $s=(\tilde{u}+\zeta_u)^{1/2}(\tilde{u}I-\tilde{\zz}%
)^{1/2}(t-\tau )$ and obtain that (\ref{LRD}) equals to
\begin{eqnarray}
&&(1+o(1))\det
(\tilde{u}I-\tilde{\mathbf{z}})^{-1/2}\tilde{u}^{-d/2}\exp
\left\{(\tilde{u}+\zeta
_{u})\left(f(\tau )+\frac{1}{2}\tilde{y}^{\top }(\tilde{u}I-\tilde{\zz})^{-1}\tilde{%
y}\right)\right\}  \notag \\
&&\times
\int_{|\tilde{u}^{-1/2}(\tilde{u}I-\tilde{\mathbf{z}})^{-1/2}s|<\lambda
_{u}}\exp \left\{-\frac{1}{2}\left| s-(\tilde{u}+\zeta _{u})^{1/2}(\tilde{u}I-%
\tilde{\mathbf{z}})^{-1/2}\tilde{y}\right| ^{2}\right\}  \notag \\
&&\times \exp \biggr\{(\tilde{u}+\zeta _{u})g_{3}\left((\tilde{u}+\zeta _{u})^{-1/2}(%
\tilde{u}I-\tilde{\zz})^{-1/2}s\right)+\tilde{u}(\tilde{u}+\zeta _{u})C_{4}\left((\tilde{u%
}+\zeta _{u})^{-1/2}(\tilde{u}I-\tilde{\zz})^{-1/2}s\right)  \notag \\
&&+(\tilde{u}+\zeta _{u})g\left((\tilde{u}+\zeta _{u})^{-1/2}(\tilde{u}I-\tilde{\zz}%
)^{-1/2}s\right)\biggr\}ds. \notag\\ \label{inter}
\end{eqnarray}%
Note that the variation of the last term in \eqref{inter}, the $g(t)$ term, is tiny. Then, we first focus on the leading term.
Similar to the proof of Lemma \ref{LemExp}, we can write the integral (without the $g(t)$ term) as
\begin{eqnarray*}
&&\int_{|\tilde{u}^{-1/2}(\tilde{u}I-\tilde{\mathbf{z}})^{-1/2}s|<\lambda
_{u}}\exp
\left\{ -\frac{1}{2}\left\vert s-(\tilde{u}+\zeta _{u})^{1/2}(\tilde{u}I-%
\tilde{\mathbf{z}})^{-1/2}\tilde{y}\right\vert ^{2}\right\} \\
&&~~~~\times \exp \biggr\{ (\tilde{u}+\zeta
_{u})g_{3}\left((\tilde{u}+\zeta
_{u})^{-1/2}(\tilde{u}I-\tilde{\zz})^{-1/2}s\right)\\
&&~~~~~~~~~~~~~~~~+\tilde{u}(\tilde{u}+\zeta
_{u})C_{4}\left((\tilde{u}+\zeta _{u})^{-1/2}(\tilde{u}I-\tilde{\zz}%
)^{-1/2}s\right)\biggr\} ds \\
&=&(1+o(1))e^{-\frac{\tilde{u}^{-2}}{8}\tilde{Y}^{\top }\mu _{22}\tilde{Y}%
}\int_{|\tilde{u}^{-1/2}(\tilde{u}I-\tilde{\mathbf{z}})^{-1/2}s|<\lambda _{u}}e^{-%
\frac{1}{2}\left\vert s-(\tilde{u}+\zeta _{u})^{1/2}(\tilde{u}I-\tilde{\mathbf{z}}%
)^{-1/2}\tilde{y}\right\vert ^{2}}ds,
\end{eqnarray*}%
where
\begin{equation*}
\tilde{Y}=\left\{\tilde{y}_{i}^{2},i=1,...,d;
2\tilde{y}_{i}\tilde{y}_{j},1\leq i<j\leq d\right\},
\end{equation*}%
is arranged in the same order as that of the elements in $Y$.
Therefore, (\ref{LRD}) equals
\begin{eqnarray}
 &&(1+o(1))\det (\tilde{u}I-\tilde{\zz})^{-1/2}\tilde{u}%
^{-d/2}\exp \left\{(\tilde{u}+\zeta _{u})\left(f(\tau )+\frac{1}{2}\tilde{y}^{\top }(%
\tilde{u}I-\tilde{\zz})^{-1}\tilde{y}\right)-\frac{\tilde{u}^{-2}}{8}\tilde{Y}^{\top
}\mu _{22}\tilde{Y}\right\}  \notag \\
&&\times
\int_{|\tilde{u}^{-1/2}(\tilde{u}I-\tilde{\mathbf{z}})^{-1/2}s|<\lambda
_{u}}\exp \left\{-\frac{1}{2}\left\vert s-(\tilde{u}+\zeta _{u})^{1/2}(\tilde{u}I-%
\tilde{\mathbf{z}})^{-1/2}\tilde{y}\right\vert ^{2}\right\}ds  \notag \\
&&\times E\left[\exp \left\{(\tilde{u}+\zeta_{u})g\left((\tilde{u}+\zeta _{u})^{-1/2}(\tilde{u}I-\tilde{\zz}%
)^{-1/2}S^{\prime }\right)\right\}\right],  \notag\\\label{cont}
\end{eqnarray}
where $S^{\prime }$ is a random variable taking values on the set $\left\{s:|%
\tilde{u}^{-1/2}(\tilde{u}I-\tilde{\zz})^{-1/2}s|<\lambda
_{u}\right\}$ with density proportional to
\begin{eqnarray*}
&&\exp \left\{ -\frac{1}{2}\left\vert s-(\tilde{u}+\zeta _{u})^{1/2}(\tilde{u%
}I-\tilde{\zz})^{-1/2}\tilde{y}\right\vert ^{2}\right\} \\
&&\times \exp \left\{ (\tilde{u}+\zeta
_{u})g_{3}\left((\tilde{u}+\zeta
_{u})^{-1/2}(\tilde{u}I-\tilde{\zz})^{-1/2}s\right)+\tilde{u}(\tilde{u}+\zeta
_{u})C_{4}\left((\tilde{u}+\zeta _{u})^{-1/2}(\tilde{u}I-\tilde{\zz}%
)^{-1/2}s\right)\right\}.
\end{eqnarray*}%
We use $\kappa$ to denote the last two terms of \eqref{cont}
\begin{eqnarray}
\kappa
&=&\int_{|\tilde{u}^{-1/2}(\tilde{u}I-\tilde{\mathbf{z}})^{-1/2}s|<\lambda
_{u}}\exp
\left\{-\frac{1}{2}\Big\vert s-(\tilde{u}+\zeta _{u})^{1/2}(\tilde{u}I-\tilde{\zz}%
)^{-1/2}\tilde{y}\Big\vert ^{2}\right\}ds  \notag\\
&&\times E\left[\exp \left\{(\tilde{u}+\zeta _{u})g\left((\tilde{u}+\zeta _{u})^{-1/2}(\tilde{u}I-\tilde{\zz}%
)^{-1/2}S^{\prime }\right)\right\} \right] \notag\\
 \label{kappa}
%&=&\left((2\pi )^{d/2}+o(1)\right) \exp\left\{O\left(\sup_{|t|\leq u^{-1/2+\delta}}|u g(t)|\right)\right\}\notag\\
%&=&\left((2\pi )^{d/2}+o(1)\right) \exp\left\{O_p\left(u^{-1/2+3\delta}\right)\right\}.\notag
\end{eqnarray}%
It is helpful it keep in mind that $\kappa = (2\pi)^{d/2} + o(1)$.
%The last step is because of $g(t) = O_p\left(|t|^3\right)$ and the Borel-TIS lemma (Lemma \ref{LemBorel}).
Now, we continue the calculations in \eqref{cont} and write (\ref{LRD}) in the form of $A$ defined  as in \eqref{A} to facilitate the change of variable later. Then, on the set $\mathcal L$, we plug in the form of $A$ and $B$ defined in Part 1 of the proof and obtain that \eqref{LRD} equals
\begin{eqnarray}
(\ref{LRD})
&=&(1+o(1))\kappa \det (\tilde{u}I-\tilde{\zz})^{-1/2}u^{-d/2}\exp \left\{ (%
\tilde{u}+\zeta _{u})\left( f(\tau )+\frac{1}{2}\tilde{y}^{\top}(\tilde{u}I-%
\tilde{z})^{-1}\tilde{y}-\frac{\tilde{u}^{-2}}{8}\tilde{Y}^{\top}\mu _{22}%
\tilde{Y}\right) \right\}   \notag \\
&=&(1+o(1))\kappa u^{-d}\exp \biggr\{\frac{\tilde{u}}{\sigma }A-\tilde{u}B+%
\frac{\tilde{u}}{2\sigma }\log \det (I-u^{-1}\zz)  \notag \\
&&+\frac{\tilde{u}}{2}\left( \tilde{y}^{\top }(\tilde{u}I-\tilde{\zz})^{-1}%
\tilde{y}-y^{\top }(uI-\zz)^{-1}y\right) +u^{-1/2+\delta }O\left(
|y|^{2}+f(\tau )\right) \biggr\}  \notag \\
%&=&(1+o(1))\kappa u^{-d}\exp \biggr\{\frac{\tilde{u}}{\sigma }A-\tilde{u}B+%
%\frac{\tilde{u}}{2\sigma }\log \det (I-u^{-1}\zz)-\tilde{y}^{\top }\partial
%\mu _{\sigma }(\tau )-\frac{1}{2}|\partial \mu _{\sigma }(\tau )|^{2}  \notag
%\\
%&&+u^{-1/2+\delta }O\left( |y|^{2}+|z|^{2}+f(\tau )\right) \biggr\}  \notag
%\\
&=&(1+o(1))\kappa u^{-d}\exp \biggr\{\frac{\tilde{u}}{\sigma }A-\tilde{u}B-%
\frac{1}{2\sigma }Tr(\tilde{\zz}+\mu _{\sigma }(\tau )I+\Delta \mu
_{\sigma }(\tau )) \notag\\
&&-\tilde{y}^{\top }\partial \mu _{\sigma }(\tau )-\frac{1}{2}|\partial \mu
_{\sigma }(\tau )|^{2}+u^{-1/2+\delta }O\left( |y|^{2}+|z|^{2}+f(\tau
)\right) \biggr\},  \notag\\
  \label{1}
\end{eqnarray}
where $Tr(\mathbf{z})$ is the trace of matrix $\mathbf{z}$. The last step in the above display is thanks to Lemma \ref{LemDet} and the fact that $y= \tilde y + \partial \mu_\sigma(\tau)$. We insert the result of \eqref{1} to the definition of $I_{2}$ and obtain that
\begin{eqnarray*}
I_{2} &=&(1+o(1))\kappa u^{-d}e^{u^{2}-u\mu _{\sigma }(\tau
)+\frac{1}{2}\mu^{2}
_{\sigma }(\tau )} \\
&&\times \exp \biggr\{\frac{\tilde{u}}{\sigma }A-\tilde{u}B-\frac{1}{2\sigma }%
Tr(\tilde{\zz}+\mu _{\sigma }(\tau )I+\Delta \mu _{\sigma }(\tau ))-%
\tilde{y}^{\top }\partial \mu _{\sigma }(\tau )-\frac{1}{2}|\partial \mu
_{\sigma }(\tau )|^{2} \\
&&+u^{-1/2+\delta }O\left(|y|^{2}+|z|^{2}+f(\tau )\right)\biggr\}.
\end{eqnarray*}%
Thanks to Lemma \ref{LemI4}, $I_3$ is of a much smaller order than $I_2$ and we obtain that
\begin{eqnarray}
K&=&I_{2}+I_{3}\notag\\
 &=&(1+o(1))\left(\kappa +O\left(e^{-\delta ^{\ast
}u^{1+2\delta }}e^{u\sup |g(t)|}\right)\right)u^{-d}e^{u^{2}-u\mu
_{\sigma }(\tau )+\frac{1}{2}\mu^{2} _{\sigma
}(\tau )} \notag\\
&&\times \exp \biggr\{\frac{\tilde{u}}{\sigma }A-\tilde{u}B-\frac{1}{2\sigma }%
Tr(\tilde{\zz}+\mu _{\sigma }(\tau )I+\Delta \mu _{\sigma }(\tau ))-%
\tilde{y}^{\top }\partial \mu _{\sigma }(\tau )-\frac{1}{2}|\partial \mu
_{\sigma }(\tau )|^{2} \notag\\
&&+u^{-1/2+\delta }O\left(|y|^{2}+|z|^{2}+f(\tau )\right)\biggr\}. \label{KK}
\end{eqnarray}

\subsection*{Part 3}

Use the notations that $f(\tau )=w$, $\partial f(\tau )=\tilde{y}$, and $%
\partial ^{2}f(\tau )=\tilde{z}$. We now put together the results from Part
1 and Part 2 and obtain an approximation of $\Lambda(\tau)$ defined as in \eqref{lambda}. Note that
\begin{eqnarray*}
\Lambda ^{\ast }(\tau ) &\triangleq&E_{\tau }^{Q}\left[ \frac{1}{\int_{T}e^{-\frac{1}{%
2}\left( f(t)-u+\mu _{\sigma }(t)\right) ^{2}+\frac{1}{2}f^{2}(t)}dt}%
;\int_{T}e^{\mu (t)+\sigma f(t)}dt>b,\mathcal{L}_{Q}\right]  \\
&=&e^{u^{2}/2}E\left[ \frac{1}{I_{2}+I_{3}};~u\cdot A>\xi _{u},\mathcal{L}%
\right]  \\
&=&e^{u^{2}/2}\int_{\mathcal{L}}E\left[ \left. \frac{1}{I_{2}+I_{3}};~u\cdot
A>\xi _{u}\right\vert f(\tau )=w,\partial f(\tau )=\tilde{y},\partial
^{2}f(\tau )=\tilde{z}\right] h(w,\tilde{y},\tilde{z})dwd\tilde{y}d\tilde{z}
\\
\end{eqnarray*}%
Plugging in \eqref{KK}, we have that
\begin{eqnarray*}%
\Lambda ^{\ast }(\tau )&=&(1+o(1))u^{d}\exp \left\{ -u^{2}/2+u\mu _{\sigma }(\tau )-\frac{1}{2}\mu
_{\sigma }^{2}(\tau )\right\}  \\
&&\times \int_{\mathcal{L}}\gamma _{u}(u\cdot A)\times \exp \biggr\{-\frac{%
\tilde{u}}{\sigma }A+\tilde{u}B+\frac{1}{2\sigma }Tr(\tilde{\zz}+\mu _{\sigma
}(\tau )I+\Delta \mu _{\sigma }(\tau )) \\
&&+\tilde{y}^{\top }\partial \mu _{\sigma }(\tau )+\frac{1}{2}|\partial \mu
_{\sigma }(\tau )|^{2}+u^{-1/2+\delta }O\left( |\tilde{y}|^{2}+|\tilde{z}%
|^{2}+w\right) \biggr\}h(w,\tilde{y},\tilde{z})dwd\tilde{y}d\tilde{z},
\end{eqnarray*}%
where $h$ is the density function of $\left(f(\tau ),\partial f(\tau
),\partial ^{2}f(\tau )\right)$ under the measure $P$. The above display also uses the fact that both $A$ and $B$ are functions of $(w,\tilde y, \tilde z)$ and therefore can be pulled outside of the conditional expectation. The notation
\begin{equation}
\gamma _{u}(x)=E\left[ \left. \frac{1}{\kappa +O(e^{-\delta ^{\ast
}u^{1+2\delta }}e^{ u\sup|g(t)|})};x>\xi _{u}\right\vert w,\tilde{y},\tilde{z}%
\right] ,  \label{gammaind}
\end{equation}%
where the expectation is taken with respect to the process $g(t)$. Note that $\mathbf 1 ^\top \tilde z = Tr(\tilde {\zz})$. Plugging in the analytic forms of $h$ (Lemma \ref{density}) and $B$ as in \eqref{B} and moving all the constants out of the integral, we obtain that
\begin{eqnarray*}
&&\Lambda^*(\tau) \\&=&(1+o(1))u^{d}e^{-u^{2}/2+u\mu _{\sigma }(\tau
)-\frac{1}{2}\mu^{2}_{\sigma }(\tau )+\frac{1}{2\sigma }Tr(\mu
_{\sigma }(\tau )I+\Delta \mu _{\sigma }(\tau
))+\frac{1}{2}|\partial \mu _{\sigma }(\tau
)|^{2}+\frac{\mathbf{1}^{\top }\mu _{22}\mathbf{1}}{8\sigma
^{2}}+\frac{1}{8\sigma ^{2}}\sum_{i}\partial
_{iiii}^{4}C(0)} \\
&&\times \frac{|\Gamma |^{-1/2}}{(2\pi )^{\frac{(d+1)(d+2)}{4}}}\int_{%
\mathcal{L}}\gamma _{u}(u\cdot A)\exp \left\{ -\frac{\tilde{u}}{\sigma }A+%
\frac{1}{2\sigma }\mathbf{1}^{\top }\tilde{z}+\tilde{y}^{\top
}\partial \mu _{\sigma
}(\tau )\right\}  \\
&&\times \exp \left\{ -\frac{1}{8}(u^{-1}Y+\mathbf 1/\sigma )^{\top }\mu
_{22}(u^{-1}Y+\mathbf 1/\sigma )-\frac{1}{2}\left[ \tilde{y}^{\top }\tilde{y}+\frac{%
(w-\mu _{20}\mu _{22}^{-1}\tilde{z})^{2}}{1-\mu _{20}\mu _{22}^{-1}\mu _{02}}%
+\tilde{z}^{\top }\mu _{22}^{-1}\tilde{z}\right] \right\}  \\
&&\times \exp \left\{u^{-1/2+\delta }O\left(|\tilde y|^{2}+|\tilde z|^{2}+w\right)\right\}%
dwd\tilde{y}d\tilde{z}.
\end{eqnarray*}
We insert%
\begin{equation*}
-\frac{1}{8}(u^{-1}Y+\mathbf 1/\sigma )^{\top }\mu _{22}(u^{-1}Y+\mathbf 1/\sigma )=-\frac{1
}{8\sigma ^{2}}\mathbf 1^{\top }\mu
_{22}\mathbf 1+u^{-1/2+\delta}O(|\tilde{y}|^{2}+1)
\end{equation*}%
into the above display. With some elementary calculation, we obtain that
\begin{eqnarray*}
&&\Lambda^*(\tau)
\\
 &=&(1+o(1))u^{d}e^{-u^{2}/2+u\mu _{\sigma }(\tau )-\frac{1}{2}\mu^{2}_{\sigma
}(\tau )+\frac{1}{2\sigma }Tr(\mu _{\sigma }(\tau )I+\Delta \mu
_{\sigma }(\tau
))+|\partial \mu _{\sigma }(\tau )|^{2}+\frac{\mathbf{1}^{\top }\mu _{22}%
\mathbf{1}}{8\sigma ^{2}}+\frac{1}{8\sigma ^{2}}\sum_{i}\partial
_{iiii}^{4}C(0)} \\
&&\frac{|\Gamma |^{-1/2}}{(2\pi )^{\frac{(d+1)(d+2)}{4}}}\int_{\mathcal{L}%
}\gamma _{u}(u\cdot A)\exp \left\{-\frac{\tilde{u}}{\sigma }A+u^{-1/2+\delta }O\left(|\tilde y|^{2}+|\tilde z|^{2}+w+1\right)\right\} \\
&&\times \exp \left\{ -\frac{1}{2}\left[ \left|\tilde{y}-\partial
\mu _{\sigma }(\tau )\right|^{2}+\frac{(w-\mu _{20}\mu
_{22}^{-1}\tilde{z})^{2}}{1-\mu _{20}\mu _{22}^{-1}\mu
_{02}}+\left|\mu _{22}^{-1/2}\tilde{z}-\mu
_{22}^{1/2}\frac{\mathbf{1}}{2\sigma} \right|^{2}\right] \right\}
dwd\tilde{y}d\tilde{z}.
\end{eqnarray*}%
Furthermore, on the set $\mathcal L$, according to the definition of $A$ in (\ref{A}) and
\begin{equation*}
w=A/\sigma +O(u^{-1/2+\varepsilon }|\tilde{y}|)+o(1),
\end{equation*}%
we obtain that
\begin{eqnarray*}
&&\Lambda^*(\tau)\\ &=&(1+o(1))u^{d}e^{-u^{2}/2+u\mu _{\sigma }(\tau
)-\frac{1}{2}\mu^{2}_{\sigma }(\tau )+\frac{1}{2\sigma }Tr(\mu
_{\sigma }(\tau )I+\Delta \mu
_{\sigma }(\tau ))+|\partial \mu _{\sigma }(\tau )|^{2}+\frac{\mathbf{1}%
^{\top }\mu _{22}\mathbf{1}}{8\sigma ^{2}}+\frac{1}{8\sigma ^{2}}%
\sum_{i}\partial _{iiii}^{4}C(0)} \\
&&\frac{|\Gamma |^{-1/2}}{(2\pi )^{\frac{(d+1)(d+2)}{4}}}\int_{\mathcal{L}%
}\gamma _{u}(u\cdot A)\exp \left\{-\frac{\tilde{u}}{\sigma }A+u^{-1/2+\delta }O\left(|\tilde y|^{2}+|\tilde z|^{2}+A\right)\right\} \\
&&\times \exp \biggr\{ -\frac{1}{2}\biggr[ \left|\tilde{y}-\partial
\mu _{\sigma }(\tau )\right|^{2}+\frac{\left(A/\sigma
+O(u^{-1/2+\varepsilon }|\tilde{y}|)+o(1)-\mu _{20}\mu
_{22}^{-1}\tilde{z}\right)^{2}}{1-\mu _{20}\mu _{22}^{-1}\mu
_{02}}\notag\\
&&~~~~~~~~~~~~~ +\left|\mu
_{22}^{-1/2}\tilde{z}-\mu _{22}^{1/2}\frac{\mathbf{1}}{2\sigma} \right|^{2}\biggr] \biggr\} dwd%
\tilde{y}d\tilde{z} \\
&=&(1+o(1))u^{d-1}e^{-\frac 1 2 (u-\mu _{\sigma
}(\tau ))^{2}+\frac{1}{2\sigma }Tr(\mu _{\sigma }(\tau )I+\Delta \mu _{\sigma
}(\tau ))+|\partial \mu _{\sigma }(\tau )|^{2}+\frac{\mathbf{1}^{\top }\mu
_{22}\mathbf{1}}{8\sigma ^{2}}+\frac{1}{8\sigma ^{2}}\sum_{i}\partial
_{iiii}^{4}C(0)} \\
&&\frac{|\Gamma |^{-1/2}}{(2\pi )^{\frac{(d+1)(d+2)}{4}}}\int_{\mathcal{L}%
}\gamma _{u}\left((\sigma +o(1))\tilde A\right)\times \exp\left\{-\tilde A+u^{-1/2+\delta }O\left(|\tilde y|^{2}+|\tilde z|^{2}+ A\right)\right\} \\
&&\times \exp \biggr\{ -\frac{1}{2}\biggr[ \left|\tilde{y}-\partial
\mu _{\sigma }(\tau )\right|^{2}+\frac{\left(\tilde
A/u+O(u^{-1/2+\varepsilon }|\tilde{y}|)+o(1)-\mu _{20}\mu
_{22}^{-1}\tilde{z}\right)^{2}}{1-\mu _{20}\mu
_{22}^{-1}\mu _{02}}\notag\\
&&~~~~~~~~~~~~~+\left|\mu _{22}^{-1/2}\tilde{z}-\mu
_{22}^{1/2}\frac{\mathbf{1}}{2\sigma} \right|^{2}\biggr] \biggr\}
d\tilde Ad\tilde{y}d\tilde{z}.
\end{eqnarray*}%
The last step changes the integral from ``$dwd\tilde y d \tilde z$'' to ``$d\tilde A d\tilde y d \tilde z$'',
where $\tilde A = \tilde u A/\sigma$.
Thanks to the Borel-TIS inequality (Lemma \ref{LemBorel}), Lemma \ref{LemRemainder} and the definition of $\kappa $ in (\ref%
{kappa}), for $x>0$, $\gamma _{u}(x)$ is bounded and as $b\rightarrow \infty
$,%
\begin{equation*}
\gamma _{u}(x)=E\left[ \frac{1}{\kappa +O(e^{-\delta ^{\ast
}u^{1+2\delta }}e^{u\sup |g(t)|})};x>\xi _{u}\right] \rightarrow
(2\pi )^{-d/2}.
\end{equation*}%
Note that on the set $\mathcal{L}$, $\tilde  A>-u^{3/2+\varepsilon }$. By Lemma \ref{LemGamma}, for $-u^{3/2+\varepsilon }<x<0$, we have that
\begin{equation*}
\gamma _{u}(x)\leq e^{u^{\delta ^{\ast }}x}.
\end{equation*}%
Therefore,%
\begin{eqnarray}
&&\gamma _{u}(x)I(x \neq 0,x>-u^{3/2+\varepsilon })  \notag \\
&&~~~~~~~~~~~~~=I(x>0)\left((2\pi )^{-d/2}+o(1)\right)+I(-u^{3/2+\varepsilon
}<x<0)O\left(e^{u^{\delta ^{\ast }}x}\right). \label{gamma}
\end{eqnarray}%
The above approximation of $\gamma _{u}(x)$ and the dominated
convergence theorem (applied to the region where $A>0$) imply that
\begin{equation*}
\int_{-u^{3/2+\varepsilon }}^{\infty }\gamma _{u}(A)\times e^{-A}~dA= \frac{1+o(1)}{(2\pi)^{d/2}}\int_0^\infty e^{-A}dA=
\frac{1+o(1)}{(2\pi )^{d/2}}.
\end{equation*}%
Then, we continue the calculation of $\Lambda^*(\tau)$ and obtain, by the dominated convergence theorem and the above results, that
\begin{eqnarray}\label{asy}\\
\Lambda (\tau )
&=&(1+o(1))\Lambda ^{\ast }(\tau )  \notag \\
&=&(1+o(1))u^{d-1}e^{-\frac{1}{2}(u-\mu _{\sigma }(\tau ))^{2}+\frac{1}{%
2\sigma }Tr(\mu _{\sigma }(\tau )I+\Delta \mu _{\sigma }(\tau ))+|\partial
\mu _{\sigma }(\tau )|^{2}+\frac{\mathbf{1}^{\top }\mu _{22}\mathbf{1}}{%
8\sigma ^{2}}+\frac{1}{8\sigma ^{2}}\sum_{i}\partial _{iiii}^{4}C(0)}  \notag
\\
&&\frac{|\Gamma |^{-1/2}}{(2\pi )^{\frac{(d+1)(d+2)}{4}}}\int \frac{I(\tilde{%
A}>0)}{(2\pi )^{d/2}}e^{-\tilde{A}}  \notag \\
&&\exp \biggr\{-\frac{1}{2}\biggr[|\tilde{y}-\partial \mu _{\sigma }(\tau
)|^{2}+\frac{\left\vert \mu _{20}\mu _{22}^{-1}\tilde{z}\right\vert ^{2}}{%
1-\mu _{20}\mu _{22}^{-1}\mu _{02}}
+\left\vert \mu _{22}^{-1/2}\tilde{z}-\mu _{22}^{1/2}\frac{\mathbf{1}}{%
2\sigma }\right\vert ^{2}\biggr]\biggr\}d\tilde{A}d\tilde{y}d\tilde{z}
\notag
\end{eqnarray}
The integrand factorizes. Then, we integrate out $d\tilde A$ and $d\tilde y$ and obtain that
\begin{eqnarray}
\Lambda (\tau)&=&
%(1+o(1))u^{d-1}e^{-\frac{1}{2}(u-\mu _{\sigma }(\tau ))^{2}+\frac{1}{%
%2\sigma }Tr(\mu _{\sigma }(\tau )I+\Delta \mu _{\sigma }(\tau ))+|\partial
%\mu _{\sigma }(\tau )|^{2}+\frac{\mathbf{1}^{\top }\mu _{22}\mathbf{1}}{%
%8\sigma ^{2}}+\frac{1}{8\sigma ^{2}}\sum_{i}\partial _{iiii}^{4}C(0)}  \notag
%\\
%&&\frac{|\Gamma |^{-1/2}}{(2\pi )^{\frac{(d+1)(d+2)}{4}}}\int \frac{1}{(2\pi
%)^{d/2}}\exp \biggr\{-\frac{1}{2}\biggr[|\tilde{y}-\partial \mu _{\sigma
%}(\tau )|^{2}+\frac{\left\vert \mu _{20}\mu _{22}^{-1}\tilde{z}\right\vert
%^{2}}{1-\mu _{20}\mu _{22}^{-1}\mu _{02}}  \notag \\
%&&~~~~~~~~~~~~~+\left\vert \mu _{22}^{-1/2}\tilde{z}-\mu _{22}^{1/2}\frac{%
%\mathbf{1}}{2\sigma }\right\vert ^{2}\biggr]\biggr\}d\tilde{y}d\tilde{z}
%\notag \\
%&=&
(1+o(1))u^{d-1}e^{-\frac{1}{2}(u-\mu _{\sigma }(\tau ))^{2}+\frac{1}{%
2\sigma }Tr(\mu _{\sigma }(\tau )I+\Delta \mu _{\sigma }(\tau ))+|\partial
\mu _{\sigma }(\tau )|^{2}+\frac{\mathbf{1}^{\top }\mu _{22}\mathbf{1}}{%
8\sigma ^{2}}+\frac{1}{8\sigma ^{2}}\sum_{i}\partial _{iiii}^{4}C(0)}  \notag
\\
&&\frac{|\Gamma |^{-1/2}}{(2\pi )^{\frac{(d+1)(d+2)}{4}}}\int_{\tilde{z}\in
R^{d(d+1)/2}}\exp \left\{ -\frac{1}{2}\left[ \frac{|\mu _{20}\mu _{22}^{-1}%
\tilde{z}|^{2}}{1-\mu _{20}\mu _{22}^{-1}\mu _{02}}+\left\vert \mu
_{22}^{-1/2}\tilde{z}-\mu _{22}^{1/2}\frac{\mathbf{1}}{2\sigma }\right\vert
^{2}\right] \right\} d\tilde{z}.  \notag
\end{eqnarray}
Thus, we conclude the situation when $\tau $ is at least $%
u^{-1/2+\delta ^{\prime }}$ away from the boundary.

\subsection*{The case in which $\protect\tau$ is close to the boundary of $T$}

For the case in which $\tau$ is within $u^{-1/2 + \delta'}$ distance from the boundary of $T$,
Lemma \ref{LemBoundary} establishes that the contribution of the boundary case is ignorable. An intuitive interpretation of Lemma \ref{LemBoundary} is that the important region of the integral $\int e^{f(t)}dt$ might be cut off by the boundary of $T$. Therefore, in cases that $\tau$ is too close to the boundary, the tail $\int e^{f(t)} dt$ is not heavier than that of the interior case.

\subsection*{Summary}

Note that $\mu_\sigma(\tau)$ does not achieve the its maximum at the boundary of $T$. Together with (\ref{asy}) and Lemma \ref{LemBoundary}, we obtain that
\begin{eqnarray*}
P\left( \int_{T}e^{\mu (t)+\sigma f(t)}dt>b\right)
&=&E^{Q}\left[ \frac{dP}{dQ};\int_{T}e^{\mu (t)+\sigma f(t)}dt>b\right]  \\
&=&\int_{T} \Lambda(\tau) d\tau  =(1+o(1))\int_{T} \Lambda^*(\tau) d\tau  \\
&=&(1+o(1)) u^{d-1} \int_T H(\mu,\sigma,\tau) \exp \left\{
-\frac{1}{2}(u-\mu _{\sigma }(\tau ))^{2}\right\} d\tau
\end{eqnarray*}%
where
\begin{eqnarray*}
&&H(\mu,\sigma,\tau)\\ &=&\frac{|\Gamma |^{-1/2}}{(2\pi
)^{\frac{(d+1)(d+2)}{4}}}\exp \left\{\frac{1}{2\sigma }Tr(\mu
_{\sigma }(\tau )I+\Delta \mu _{\sigma }(\tau ))+|\partial \mu
_{\sigma }(\tau )|^{2}+ \frac{\mathbf{1}^{\top }\mu
_{22}\mathbf{1}}{8\sigma ^{2}}+\frac{1}{8\sigma
^{2}}\sum_{i}\partial _{iiii}^{4}C(0)\right\}  \\
&&\times \int_{\tilde{z}\in R^{d(d+1)/2}} \exp \left\{ -\frac{1}{2}\left[ \frac{|\mu _{20}\mu _{22}^{-1}%
\tilde{z}|^{2}}{1-\mu _{20}\mu _{22}^{-1}\mu _{02}}+\left|\mu _{22}^{-1/2}\tilde{z%
}-\mu _{22}^{1/2}\frac{\mathbf{1}}{2\sigma} \right|^{2}\right]
\right\} d\tilde{z}.
\end{eqnarray*}

\section{Lemmas}\label{SecLem}

In this section, we state all the lemmas used in
the previous section. To facilitate reading, we move several lengthy proofs to the supplemental article \cite{LXsuppA}.

The first lemma is known as Borel-TIS lemma, which was proved
independently by \cite{Bor75, CIS}.

\begin{lemma}[Borel-TIS]\label{LemBorel}
Let $f(t)$, $t\in \mathcal U$, $\mathcal U$ is a parameter set,
be mean zero Gaussian random field. $f$ is almost surely
bounded on $\mathcal U$. Then,
$$E(\sup_{\mathcal U}f(t) )<\infty,$$ and
\[
P(\max_{t\in \mathcal{U}}f\left(  t\right)
-E[\max_{t\in\mathcal{U}}f\left( t\right)  ]\geq b)\leq e^{
-\frac{b^{2}}{2\sigma_{\mathcal{U}}^{2}}}  ,
\]
where
\[
\sigma_{\mathcal{U}}^{2}=\max_{t\in \mathcal{U}}Var[f( t)].
\]
\end{lemma}

\begin{lemma}
\label{LemLocal} For each $\varepsilon >0$, let $\mathcal{L}_{Q}$ be
as defined in \eqref{LQ}. There exists some $\lambda >0$, such that
\begin{eqnarray*}
E_{\tau }^{Q}\left[ \frac{1}{\int_T \exp \left\{ (u-\mu _{\sigma
}(t))(f(t)+\mu _{\sigma }(t))+\frac{1}{2}\mu _{\sigma }^{2}(t)\right\} dt}%
;\int_{T}e^{\mu (t)+\sigma f(t)}dt>b,\mathcal{L}_{Q}^{c}\right] =o(1)e^{-u^{2}-u^{1+\lambda }},
\end{eqnarray*}%
where $\mathcal{L}_{Q}^{c}$ is the complement of set $\mathcal{L}_{Q}$.
\end{lemma}

\begin{lemma}
\label{LemRemainder} Let $\xi _{u}$ be as defined in (\ref{xi}),
then there exists a $\lambda >0$ such that for all $x>0$
\begin{equation*}
P\left( u^{1/2-3\delta }|\xi _{u}|>x\right) \leq e^{-\lambda x^{2}}+
e^{-\lambda u^2},
\end{equation*}
for $u$ sufficiently large.
\end{lemma}

\begin{proof}[Proof of Lemma \protect\ref{LemRemainder}]
We split the expectation into two parts $\{|\tilde S|\leq u^{\delta}\}$ and $\{|\tilde S|> u^\delta, \tau + (uI-\zz)^{-1/2}\tilde S \in T\}$.
Note that $|S|\leq \kappa u^{\delta }$ and $g(t)$ is a mean zero Gaussian random field with $Var(g(t))=O(|t|^{6})$. A direct application
of the Borel-TIS inequality (Lemma \ref{LemBorel}) yields the result of this lemma.
\end{proof}

\begin{lemma}
\label{LemExp} Let $S$ be a random variable taking values in
$\left\{s:(uI-\zz)^{-1/2}s+\tau \in T\right\}$ with density
proportional to
\begin{equation*}
\exp \left\{ -\frac{\sigma }{2}\left(s-(uI-\zz)^{-1/2}y\right)^{\top }\left(s-(uI-\zz)^{-1/2}y\right)%
\right\}.
\end{equation*}%
Then, on the set $\mathcal L$
\begin{eqnarray*}
&&\log \left\{ E\exp \left[ \sigma
g_{3}\left((uI-\zz)^{-\frac{1}{2}}S\right)+\sigma
\left(u-\mu _{\sigma }(\tau )\right)C_{4}\left((uI-\zz)^{-\frac{1}{2}}S\right)+\sigma R\left((uI-\zz)^{-%
\frac{1}{2}}S\right)\right] \right\} \\
&=&-\frac{\sigma }{8u}\left(u^{-1}Y+\mathbf{1}/\sigma \right)^{\top }\mu _{22}\left(u^{-1}Y+%
\mathbf{1}/\sigma \right)+\frac{\mathbf{1}^{\top }\mu _{22}\mathbf{1}}{8\sigma u}+%
\frac{1}{8\sigma u}\sum_{i}\partial
_{iiii}^{4}C(0)+o\left(u^{-1}\right),
\end{eqnarray*}
where  the  expectation is taken with respect to $S$  as in
\eqref{l4}.
\end{lemma}

\begin{lemma}
\label{LemI4} Let $I_{3}$ be as defined in the Part 2 of the proof
of Theorem \ref{ThmG}. On the set $\mathcal{L}$, there exists some
$\delta ^{\ast }>0$ such that
\begin{eqnarray*}
I_{3} &=&O\left(e^{-\delta ^{\ast }u^{1+2\delta }}\right)e^{u\sup
|g(t)|}u^{-d}e^{u^{2}-u\mu _{\sigma }(\tau )+\frac{1}{2}\mu^{2}_{\sigma
}(\tau)} \\
&&\times \exp \biggr\{\frac{\tilde{u}}{\sigma
}A-\tilde{u}B-\frac{1}{2\sigma } Tr(\Delta f(\tau )+\mu _{\sigma
}(\tau )I+\Delta \mu _{\sigma }(\tau ))- \tilde{y}^{\top }\partial
\mu _{\sigma }(\tau )-\frac{1}{2}|\partial \mu
_{\sigma }(\tau )|^{2} \\
&&+u^{-1/2+\delta }O(|y|^{2}+|z|^{2}+f(\tau ))\biggr\}.
\end{eqnarray*}
\end{lemma}

\begin{lemma}\label{LemDet}
$$\log(\det (I-u^{-1}\zz)) = - u^{-1}Tr(\zz) + \frac 1 2
u^{-2}\mathcal I_2(\zz) + o(u^{-2}),$$ where $Tr$ is the trace
of a matrix, $\mathcal I_2 (\zz)=\sum_{i=1}^d \lambda_i^2$, and
$\lambda_i$'s are the eigenvalues of $\zz$.
\end{lemma}
\begin{proof}[Proof of Lemma \ref{LemDet}] The result is immediate by noting that
\begin{eqnarray*}
\det (I-u^{-1}\zz)&=&\prod_{i=1}^{d}(1-\lambda _{i}/u),
\end{eqnarray*}
and $Tr(\zz)= \sum_{i=1}^d \lambda_i$.
\end{proof}

\begin{lemma}\label{density} For the homogeneous Gaussian random field ${f}(t)$ in Theorem \ref{ThmG},
let $h(w,\tilde{y},\tilde{z})$ be the density of $(f(\tau)$,
$\partial{f}(\tau)$, $\partial^2 {f}(\tau))$ evaluated at
$(w,\tilde{y},\tilde{z})$. Then,
\begin{eqnarray}
~~~~~~~h(w,\tilde{y},\tilde{z})&=&
\frac{|\Gamma|^{-\frac{1}{2}}}{(2\pi)^\frac{(d+1)(d+2)}{4}}\exp\biggr\{
-\frac{1}{2}\biggr[\tilde{y}^{\top
}\tilde{y}+\frac{(w-\mu_{20}\mu_{22}^{-1}\tilde{z})^2}{1-\mu_{20}\mu_{22}^{-1}\mu_{02}}+\tilde{z}^{\top
}\mu_{22}^{-1}\tilde{z}\biggr]\biggr\},
\end{eqnarray}
where $\Gamma$ is the covariance matrix of  $({f}(\tau),
\partial^2 {f}(\tau))$ whose inverse is
\begin{equation}\label{inverse}
 \Gamma^{-1}=\left(
\begin{array}{cc}
 \frac{1}{1-\mu_{20}\mu_{22}^{-1}\mu_{02}} & \frac{-\mu_{20}\mu_{22}^{-1}}{1-\mu_{20}\mu_{22}^{-1}\mu_{02}}\\
 \frac{-\mu_{20}\mu_{22}^{-1}}{1-\mu_{20}\mu_{22}^{-1}\mu_{02}} & \mu_{22}^{-1}+\frac{\mu_{22}^{-1}\mu_{02}\mu_{20}\mu_{22}^{-1}}{1-\mu_{20}\mu_{22}^{-1}\mu_{02}}
\end{array}%
\right).
\end{equation}
\end{lemma}

\begin{proof}[Proof of Lemma \ref{density}]
The form of (\ref{inverse}) results from direct application of the
block matric inverse of linear algebra. Note that
$$h(w,\tilde{y},\tilde{z})=\frac{1}{(2\pi)^\frac{(d+1)(d+2)}{4}}|\Gamma|^{-\frac{1}{2}}\exp\left\{-\frac{1}{2}(w, \tilde{z}^{\top }, \tilde{y}^{\top})
\left({\Gamma^{-1} \atop 0}\; {0 \atop I}\right)(w,
\tilde{z}^{\top}, \tilde{y}^{\top })^{\top }\right\}.$$ By plugging
in the form of $\Gamma^{-1}$, we get the conclusion.
\end{proof}

\begin{lemma}
\label{LemGamma} Consider that $\{t:|t-\tau |\leq u^{-1/2+\delta ^{\prime }}\}\subset T$. Let $\gamma _{u}(x)$ be as defined in
(\ref{gammaind}). There
exists some $\delta ^{\ast }>0$ such that for all $0<x<u^{3/2+\varepsilon }$%
,
\begin{equation*}
\gamma _{u}(-x)\leq e^{-u^{\delta ^{\ast }}x}.
\end{equation*}
\end{lemma}

\begin{lemma}\label{LemBoundary}
For each that $\tau$ within $u^{-1/2 + \delta'}$ distance from the boundary of $T$, that is, there exists an $s$ such that $|s-\tau |\leq u^{-1/2+\delta'}$ and $s\notin T$, we have that
\begin{eqnarray}
E_{\tau }^{Q}\left[ \frac{dP}{dQ};\int_{T}e^{\mu (t)+\sigma f(t)}dt>b,%
\mathcal{L}_{Q}\right]   =O(1)u^{d-1}e^{-\frac{1}{2}(u-\mu _{\sigma }(\tau ))^{2}}.
\notag
\end{eqnarray}
\end{lemma}

\begin{supplement}[id=suppA]
%  \sname{}
  \stitle{Proofs of several lemmas in Section \ref{SecLem} and the numerical results}
  \slink[doi]{.}
%  \slink[url]{http://lib.stat.cmu.edu/aoas/./.}
%  \sdatatype{.pdf}
  \sdescription{This supplement contains proofs of Lemmas \protect\ref{LemLocal}, \protect\ref{LemExp}, \protect\ref{LemI4}, \protect\ref{LemGamma}, and \ref{LemBoundary} as well as  numerical results.}
\end{supplement}

\bibliographystyle{plain}
\bibliography{bibstat,bibprob,RefGrant}

\begin{thebibliography}{10}

\bibitem{Adl81}
R.J. Adler.
\newblock {\em The Geometry of Random Fields}.
\newblock Wiley, Chichester, U.K.; New York, U.S.A., 1981.

\bibitem{ABL08}
R.J. Adler, J.H. Blanchet, and J.C. Liu.
\newblock Efficient simulation for tial probabilities of gaussian random
  fields.
\newblock In {\em Proceeding of Winter Simulation Conference}, 2008.

\bibitem{ABL09}
R.J. Adler, J.H. Blanchet, and J.C. Liu.
\newblock Efficient monte carlo for large excursions of gaussian random fields.
\newblock {\em Preprint}, 2009.

\bibitem{AMR}
R.J. Adler, P.~Muller, and B.~Rozovskii, editors.
\newblock {\em Stochastic Modelling in Physical Oceanography}.
\newblock Birkha\"user, Boston, 1996.

\bibitem{AST09}
R.J. Adler, G.~Samorodnitsky, and J.E. Taylor.
\newblock High level excursion set geometry for non-gaussian infinitely
  divisible random fields.
\newblock {\em preprint}, 2009.

\bibitem{AdlTay07}
R.J. Adler and J.E. Taylor.
\newblock {\em Random fields and geometry}.
\newblock Springer, 2007.

\bibitem{ATW09}
R.J. Adler, J.E. Taylor, and K.J. Worsley.
\newblock {\em Applications of Random Fields and Geometry: Foundations and Case
  Studies}.
\newblock preprint, 2009.

\bibitem{Ahs78}
S.M. Ahsan.
\newblock Portfolio selection in a lognormal securities market.
\newblock {\em Zeitschrift Fur Nationalokonomie-Journal of Economics},
  38(1-2):105--118, 1978.

\bibitem{AW05}
J.~M. Azais and M.~Wschebor.
\newblock On the distribution of the maximum of a gaussian field with d
  parameters.
\newblock {\em Annals of Applied Probability}, 15(1A):254--278, 2005.

\bibitem{AW08}
J.~M. Azais and M.~Wschebor.
\newblock A general expression for the distribution of the maximum of a
  gaussian field and the approximation of the tail.
\newblock {\em Stochastic Processes and Their Applications}, 118(7):1190--1218,
  2008.

\bibitem{AW09}
J.~M. Azais and M.~Wschebor.
\newblock {\em Level sets and extrema of random processes and fields}.
\newblock Wiley, Hoboken, N.J., 2009.

\bibitem{BasSha01}
S.~Basak and A.~Shapiro.
\newblock Value-at-risk-based risk management: Optimal policies and asset
  prices.
\newblock {\em Review of Financial Studies}, 14(2):371--405, 2001.

\bibitem{Berman85}
S.~M. Berman.
\newblock An asymptotic formula for the distribution of the maximum of a
  gaussian process with stationary increments.
\newblock {\em Journal of Applied Probability}, 22(2):454--460, 1985.

\bibitem{BlaSch73}
F.~Black and M.~Scholes.
\newblock Pricing of options and corporate liabilities.
\newblock {\em Journal of Political Economy}, 81(3):637--654, 1973.

\bibitem{BLY}
J.~H. Blanchet, J.~Liu, and X.~Yang.
\newblock Monte carlo for large credit portfolios with potentially high
  correlations.
\newblock {\em Proceedings of the 2010 Winter Simulation Conference}, 2010.

\bibitem{Bor75}
C.~Borell.
\newblock The {B}runn-{M}inkowski inequality in {G}auss space.
\newblock {\em Inventiones Mathematicae}, 1975.

\bibitem{Bor03}
C.~Borell.
\newblock The {Ehrhard} inequality.
\newblock {\em Comptes Rendus Mathematique}, 337(10):663--666, 2003.

\bibitem{Camp94}
M.~J. Campbell.
\newblock Time-series regression for counts - an investigation into the
  relationship between sudden-infant-death-syndrome and
  environmental-temperature.
\newblock {\em Journal of the Royal Statistical Society Series a-Statistics in
  Society}, 157:191--208, 1994.

\bibitem{TaylorWorsley082}
N.~Chamandy, K.J. Worsley, J.~Taylor, and F.~Gosselin.
\newblock Tilted euler characteristic densities for central limit random
  fields, with applications to `bubbles'.
\newblock {\em Annals of Statistics}, 36(5):2471--2507, 2008.

\bibitem{ChLe95}
K.~S. Chan and J.~Ledolter.
\newblock Monte-carlo em estimation for time-series models involving counts.
\newblock {\em Journal of the American Statistical Association},
  90(429):242--252, 1995.

\bibitem{CHGE92}
G.~Christakos.
\newblock {\em Random field models in earth sciences}.
\newblock Academic Press, San Diego, 1992.

\bibitem{CHGE00}
G.~Christakos.
\newblock {\em Modern spatiotemporal geostatistics}.
\newblock Studies in mathematical geology. Oxford University Press, 2000.

\bibitem{COX55}
D.~R. Cox.
\newblock Some statistical methods connected with series of events.
\newblock {\em Journal of the Royal Statistical Society Series B-Statistical
  Methodology}, 17(2):129--164, 1955.

\bibitem{COIS80}
D.~R. Cox and Valerie Isham.
\newblock {\em Point processes}.
\newblock Monographs on applied probability and statistics. Chapman and Hall,
  London ; New York, 1980.

\bibitem{DARO91}
R.~Daley.
\newblock {\em Atmospheric data analysis}.
\newblock Cambridge atmospheric and space science series 2. Cambridge
  University Press, Cambridge ; New York, 1991.

\bibitem{DDW00}
R.~A. Davis, W.~T.~M. Dunsmuir, and Y.~Wang.
\newblock On autocorrelation in a poisson regression model.
\newblock {\em Biometrika}, 87(3):491--505, 2000.

\bibitem{Dembo}
A.~Dembo and O.~Zeitouni.
\newblock {\em Large deviations techniques and applications}.
\newblock Springer-Verlag. New York., 1998.

\bibitem{Due04}
H.~P. Deutsch.
\newblock {\em Derivatives and internal models}.
\newblock Finance and capital markets. Palgrave Macmillan, Houndmills,
  Basingstoke, Hampshire ; New York, N.Y., 3rd edition, 2004.

\bibitem{DufPan97}
D.~Duffie and J.~Pan.
\newblock An overview of value at risk.
\newblock {\em The Journal of Derivatives}, 4(3):7--49, 1997.

\bibitem{Duf01}
D.~Dufresne.
\newblock The integral of geometric brownian motion.
\newblock {\em Advances in Applied Probability}, 33(1):223--241, 2001.

\bibitem{GHS00}
P.~Glasserman, P.~Heidelberger, and P.~Shahabuddin.
\newblock Variance reduction techniques for estimating value-at-risk.
\newblock {\em Management Science}, 46(10):1349--1364, 2000.

\bibitem{GHVKP08}
J.~R. Gott, D.~C. Hambrick, M.~S. Vogeley, J.~Kim, C.~Park, Y.~Y. Choi, R.~Cen,
  J.~P. Ostriker, and K.~Nagamine.
\newblock Genus topology of structure in the sloan digital sky survey: Model
  testing.
\newblock {\em Astrophysical Journal}, 675(1):16--28, 2008.

\bibitem{HRT}
M.S. Handcock, A.E. Raftery, and J.M. Tantrum.
\newblock Model-based clustering for social network.
\newblock {\em Annals of Applied Statistics}, 170.

\bibitem{HD02}
S.~R. Hanna and J.~M. Davis.
\newblock Evaluation of a photochemical grid model using estimates of
  concentration probability density functions.
\newblock {\em Atmospheric Environment}, 36(11):1793--1798, 2002.

\bibitem{HRH02}
P.~D. Hoff, A.~E. Raftery, and M.~S. Handcock.
\newblock Latent space approaches to social network analysis.
\newblock {\em Journal of the American Statistical Association},
  97(460):1090--1098, 2002.

\bibitem{Hu90}
J\"urg H\"usler.
\newblock Extreme values and high boundary crossings of locally stationary
  gaussian processes.
\newblock {\em Annals of Probability}, 18:1141--1158, 1990.

\bibitem{HPZ11}
J\"urg H\"usler, Vladimir Piterbarg, and Yueming Zhang.
\newblock Extremes of gaussian processes with random variance.
\newblock {\em Electronic Journal of Probability}, 16(45):1254--1280, 2011.

\bibitem{LS70}
H.~J. Landau and L.~A. Shepp.
\newblock Supremum of a gaussian process.
\newblock {\em Sankhya-the Indian Journal of Statistics Series A},
  32(Dec):369--378, 1970.

\bibitem{LT91}
M.~Ledoux and M.~Talagrand.
\newblock Probability in banach spaces : isoperimetry and processes.
\newblock 1991.

\bibitem{Liu10}
J.~Liu.
\newblock Tail approximations of integrals of gaussian random fields.
\newblock {\em Annals of Probability}, to appear, 2011.

\bibitem{LXsuppA}
J.~Liu and G.~Xu.
\newblock Supplement to ``some asymptotic results of gaussian random fields
  with varying mean functions and the associated processes''.
\newblock {\em DOI:}, 2011.

\bibitem{MS70}
M.~B. Marcus and L.~A. Shepp.
\newblock Continuity of gaussian processes.
\newblock {\em Transactions of the American Mathematical Society}, 151(2),
  1970.

\bibitem{fMER73a}
R.~Merton.
\newblock The theory of rational option pricing.
\newblock {\em Bell Journal of Economics and Management Science}, 4:141--183,
  1973.

\bibitem{NSY08}
Y.~Nardi, D.~O. Siegmund, and B.~Yakir.
\newblock The distribution of maxima of approximately gaussian random fields.
\newblock {\em Annals of Statistics}, 36(3):1375--1403, 2008.

\bibitem{Pit96}
V.~I. Piterbarg.
\newblock {\em Asymptotic methods in the theory of Gaussian processes and
  fields}.
\newblock American Mathematical Society, Providence, R.I., 1996.

\bibitem{RabinSiegmund97}
D.~Rabinowitz and D.~Siegmund.
\newblock The approximate distribution of the maximum of a smoothed poisson
  random field.
\newblock {\em Statistica Sinica}, 7:167--180, 1997.

\bibitem{Rubin}
Y.~Rubin.
\newblock {\em Applied Stochastic Hydrogeology}.
\newblock Oxford Univ.\ Press, New York, 2002.

\bibitem{SSSW}
K.~Shafie, B.~Sigal, D.~Siegmund, and K.~J. Worsley.
\newblock Rotation space random fields with an application to f{MRI} data.
\newblock {\em Ann. Statist.}, 31:1732--1771, 2003.

\bibitem{SIE76}
D.~Siegmund.
\newblock Importance sampling in the {M}onte {C}arlo study of sequential tests.
\newblock {\em Ann.\ Stat.}, 4:673--684, 1976.

\bibitem{SY00}
D.~Siegmund and B.~Yakir.
\newblock Tail probabilities for the null distribution of scanning statistics.
\newblock {\em Bernoulli}, 6(2):191--213, 2000.

\bibitem{SGD93}
G.~Smoot and K.~Davidson.
\newblock {\em Wrinkles in time}.
\newblock W. Morrow, New York, 1993.

\bibitem{Sni02}
T.A.B. Snijders.
\newblock Markov chain monte carlo estimation of exponential random graph
  models.
\newblock {\em Journal of Social Structur}, 3, 2002.

\bibitem{stein}
M.L. Stein.
\newblock {\em Interpolation of Spatial Data: Some Theory for Kriging}.
\newblock Springer series in statistics. Springer, New York, 1999.

\bibitem{ST74}
V.N. Sudakov and B.S. Tsirelson.
\newblock Extremal properties of half spaces for spherically invariant
  measures.
\newblock {\em Zap. Nauchn. Sem. LOMI}, 45:75--82, 1974.

\bibitem{Sun93}
J.~Y. Sun.
\newblock Tail probabilities of the maxima of gaussian random-fields.
\newblock {\em Annals of Probability}, 21(1):34--71, 1993.

\bibitem{TA96}
M.~Talagrand.
\newblock Majorizing measures: The generic chaining.
\newblock {\em Annals of Probability}, 24(3):1049--1103, 1996.

\bibitem{TTA05}
J.~Taylor, A.~Takemura, and R.~J. Adler.
\newblock Validity of the expected euler characteristic heuristic.
\newblock {\em Annals of Probability}, 33(4):1362--1396, 2005.

\bibitem{TaylorWorsley08}
J.~Taylor and K.~J. Worsley.
\newblock Random fields of multivariate test statistics, with an application to
  shape analysis.
\newblock {\em Annals of Statistics}, 36(1):1--27, 2008.

\bibitem{Taylor-Worsley-JASA}
J.~E. Taylor and K.~J. Worsley.
\newblock Detecting sparse signals in random fields, with an application to
  brain mapping.
\newblock {\em J.\ Amer.\ Statist.\ Assoc.}, 479:913--928, 2007.

\bibitem{CIS}
B.S. Tsirelson, I.A. Ibragimov, and V.N. Sudakov.
\newblock Norms of {G}aussian sample functions.
\newblock {\em Proceedings of the Third Japan-USSR Symposium on Probability
  Theory (Tashkent, 1975)}, 550:20--41, 1976.

\bibitem{WT06}
K.~J. Worsley and J.~E. Taylor.
\newblock Detecting {fMRI} activation allowing for unknown latency of the
  hemodynamic response.
\newblock {\em Neuroimage}, 29(2):649--654, 2006.

\bibitem{XFS10}
E.~P. Xing, W.~J. Fu, and L.~Song.
\newblock A state-space mixed membership blockmodel for dynamic network
  tomography.
\newblock {\em Annals of Applied Statistics}, 4(2):535--566, 2010.

\bibitem{Yor92}
M.~Yor.
\newblock On some exponential functionals of brownian-motion.
\newblock {\em Advances in Applied Probability}, 24(3):509--531, 1992.

\bibitem{Zeger88}
S.~L. Zeger.
\newblock A regression-model for time-series of counts.
\newblock {\em Biometrika}, 75(4):621--629, 1988.

\end{thebibliography}

\newpage
 \setcounter{page}{1} 
\center{\bf\Large{Supplement to Paper ``Some Asymptotic Results of Gaussian Random Fields with Varying Mean Functions and the Associated Processes''\\}}

\center{\large{By Jingchen Liu*  and Gongjun Xu\\}}

\center{Columbia University\\}
\footnotetext[1]{*Research supported in part by Institute of Education Sciences, through Grant R305D100017, NSF CMMI-1069064, and NSF SES-1123698}

\bigskip 
\quote{ This supplement contains proofs of Lemmas \protect\ref{LemLocal}, \protect\ref{LemExp}, \protect\ref{LemI4}, \protect\ref{LemGamma} and \ref{LemBoundary} and a subsection containing numerical results to illustrate the approximation results in Section \ref{SecMain}.\\}

\bigskip 
\centerline{\large{\bf A: Proofs of Lemmas}}

\begin{proof}[Proof of Lemma \protect\ref{LemLocal}]
Note that
\begin{equation*}
\frac{1}{mes(T)}\int_{T}e^{\sigma \left( \mu _{\sigma }(t)+f(t)\right) }dt>
\frac{1}{mes(T)}b
\end{equation*}
implies that for large $b$,
\begin{eqnarray*}
&&\frac{1}{mes(T)}\int_{T}\exp \left\{ (u-\mu _{\sigma }(t))\cdot
(f(t)+\mu
_{\sigma }(t))+\frac{1}{2}\mu^{2}_{\sigma }(t)\right\} dt \\
&\geq & \exp \left\{ \frac{1}{2}\min_{t\in T}\mu^{2}_{\sigma
}(t)\right\}\frac{1}{mes(T)}\biggr[\int_{T\cap\{f(t)+\mu_\sigma(t)\geq
0\}}\exp \left\{ \left(u-\max_{t\in T}\mu_{\sigma
}(t)\right)(f(t)+\mu _{\sigma
}(t))\right\} dt\\
&&+\int_{T\cap\{f(t)+\mu_\sigma(t)< 0\}}\exp \left\{
\left(u-\min_{t\in T}\mu_{\sigma}(t)\right)
(f(t)+\mu _{\sigma }(t))\right\} dt\biggr] \\
&\geq &\frac{1}{2}\exp \left\{ \frac{1}{2}\min_{t\in
T}\mu^{2}_{\sigma }(t)\right\} \left[\frac{1}{mes(T)}\int_{T}\exp
\left\{ \left(u-\max_{t\in T}\mu_{\sigma
}(t)\right)(f(t)+\mu _{\sigma }(t))\right\} dt\right] \\
\end{eqnarray*}
By Jensen's inequality, the above display is
\begin{eqnarray*}
&\geq &\frac{1}{2}\exp \left\{ \frac{1}{2}\min_{t\in T}\mu^{2}_{\sigma
}(t)\right\} \left[ \frac{1}{mes(T)}b\right]
^{\left(u-\max_{t\in T}\mu_{\sigma
}(t)\right)/\sigma } \\
&\geq &\exp \left\{ u^{2}-c_{1}u\log u-c_{2}\right\},
\end{eqnarray*}
where $c_1$ and $c_2$ are some constants. The last step of the above display uses the fact that $u - \log b  = O(\log u)$.

Then for any $\varepsilon>0$,
\begin{eqnarray*}
&&E_{\tau }^{Q}\left[ \frac{1}{\int_T \exp \left\{ (u-\mu _{\sigma
}(t))(f(t)+\mu _{\sigma }(t))+\frac{1}{2}\mu _{\sigma
}^{2}(t)\right\} dt} ;\int_{T}e^{\mu (t)+\sigma f(t)}dt>b,|f(\tau
)-u|>u^{1/2+\varepsilon }\right]
\\
&=& \frac{1}{mes(T)}E_{\tau }^{Q}\biggr[\frac{1}{\frac{1}{mes(T)}
\int_{T}\exp \left( (u-\mu _{\sigma }(t))\cdot (f(t)+\mu _{\sigma
}(t))+
\frac{1}{2}\mu^2 _{\sigma }(t)\right) dt}; \\
&&\frac{1}{mes(T)}\int_{T}e^{\mu (t)+\sigma f(t)}dt>\frac{1}{mes(T)}
b,|f(\tau )-u|>u^{1/2+\varepsilon }\biggr] \\
&\leq &\exp \left\{ -u^{2}+c_{1}u\log u+c_{2}\right\} Q\left(
|f(\tau )-u|>u^{1/2+\varepsilon }\Big|\tau \right) \\
&=&o(1)e^{-u^{2}-u^{1+\lambda }}.
\end{eqnarray*}
Similarly, we can get
\begin{equation*}
E_{\tau }^{Q}\left[ \frac{1}{\int_T \exp \left\{ (u-\mu _{\sigma
}(t))(f(t)+\mu _{\sigma }(t))+\frac{1}{2}\mu _{\sigma
}^{2}(t)\right\} dt};\int_{T}e^{\mu (t)+\sigma
f(t)}dt>b,\mathcal{L}_{Q}^{c}\right] =o(1)e^{- u^{2}-u^{1+\lambda
}}.
\end{equation*}
\end{proof}

\begin{proof}[Proof of Lemma \protect\ref{LemExp}]
For $\delta >\varepsilon $, we first split the expectation into two
parts:
\begin{eqnarray*}
&&E\left[\exp \left\{ \sigma
g_{3}\left((uI-\zz)^{-\frac{1}{2}}S\right)+\sigma (u-\mu _{\sigma
}(\tau ))C_{4}\left((uI-\zz)^{-\frac{1}{2}}S\right)+\sigma R\left((uI-\zz)^{-\frac{1}{2}}S\right)%
\right\}\right]\\
&=&E\left[ ...;|S|\leq u^{\delta }\right] +E\left[ ...;|S|>u^{\delta
}\right] \\
&=&J_{1}+J_{2}.
\end{eqnarray*}%

Let $t=(uI-\zz)^{-\frac{1}{2}}s$. Given that $C(t)$ is a monotone
non-increasing function, we have that for all $|s|= O(u^{\delta })$
\begin{eqnarray*}
&&-\frac{\sigma }{2}\left(s-(uI-\zz)^{-1/2}y\right)^{\top }\left(s-(uI-\zz)^{-1/2}y\right) \\
&&+\sigma g_{3}\left((uI-\zz)^{-\frac{1}{2}}s\right)+\sigma (u-\mu
_{\sigma }(\tau
))C_{4}\left((uI-\zz)^{-\frac{1}{2}}s\right)+\sigma R\left((uI-\zz)^{-\frac{1}{2}}s\right) \\
&=&-\frac{\sigma }{2}\left(t-(uI-\zz)^{-/2}y\right)^{\top }(uI-\zz)\left(t-(uI-\zz)^{-1/2}y\right) \\
&&+\sigma g_{3}(t)+\sigma (u-\mu _{\sigma }(\tau ))C_{4}(t)+\sigma R(t) \\
&\leq &-\lambda u|t|^{2}=-\lambda s^{2}.
\end{eqnarray*}%
Therefore, for $\lambda'$ small
\begin{equation*}
J_{2}=O\left(e^{-\lambda'u^{2\delta }}\right).
\end{equation*}%

We now consider the leading term $J_{1}$. One the set that $|S|\leq
u^{\delta }$ and $|y|\leq u^{1/2+\varepsilon }$, we have that%
\begin{equation*}
\sigma g_{3}\left((uI-\zz)^{-\frac{1}{2}}S\right)+\sigma (u-\mu
_{\sigma }(\tau
))C_{4}\left((uI-\zz)^{-\frac{1}{2}}S\right)+\sigma R\left((uI-\zz)^{-\frac{1}{2}%
}S\right)=O\left(u^{-1+4\delta }\right).
\end{equation*}%
By using Taylor's expansion twice, we can essentially move the expectation
into the exponent and obtain that $\log J_{1}$ equals %
\begin{equation*}
E\left[ \sigma g_{3}\left((uI-\zz)^{-\frac{1}{2}}S\right)+\sigma
(u-\mu
_{\sigma }(\tau ))C_{4}\left((uI-\zz)^{-\frac{1}{2}}S\right)+\sigma R\left((uI-\zz)^{-\frac{1}{2}%
}S\right);|S|\leq u^{\delta }\right] +o\left(u^{-1}\right).
\end{equation*}%
Note that $(uI-\zz)^{-1/2}y=u^{-1/2}y+O(u^{-3/2}|y||z|)$ and
$y=\tilde{y}+O(1)$. Let $Z=(Z_{1},...,Z_{d})$ be a multivariate
Gaussian random vector with mean zero and covariance function
$\sigma ^{-1}I$. Then, $S$ is equal in distribution to $Z+u^{-1/2}y
+O(u^{-3/2}|y||z|)$. Therefore, we obtain that
\begin{eqnarray*}
&&\log J_{1}\\ &=&-\frac{\sigma +O(u^{-1+\varepsilon })}{6}u^{-3/2}\sum_{ijkl}%
\partial^4_{ijkl}C(0)E\left[
(u^{-1/2}y_{i}+Z_{i})(u^{-1/2}y_{j}+Z_{j})(u^{-1/2}y_{k}+Z_{k})\right]
\tilde{y}_{l} \\
&&+\frac{\sigma +O(u^{-1+\varepsilon
})}{24}u^{-1}\sum_{ijkl}\partial^4 _{ijkl}C(0)E\left[
(u^{-1/2}y_{i}+Z_{i})(u^{-1/2}y_{j}+Z_{j})(u^{-1/2}y_{k}+Z_{k})(u^{-1/2}y_{l}+Z_{l})%
\right] \\
&&+o\left(u^{-1}\right) \\
&=&-\frac{\sigma }{6}u^{-3/2}\sum_{ijkl}\partial^4
_{ijkl}C(0)E\left[
(u^{-1/2}y_{i}+Z_{i})(u^{-1/2}y_{j}+Z_{j})(u^{-1/2}y_{k}+Z_{k})\right]
y_{l}
\\
&&+\frac{\sigma }{24}u^{-1}\sum_{ijkl}\partial^4 _{ijkl}C(0)E\left[
(u^{-1/2}y_{i}+Z_{i})(u^{-1/2}y_{j}+Z_{j})(u^{-1/2}y_{k}+Z_{k})(u^{-1/2}y_{l}+Z_{l})%
\right] +o\left(u^{-1}\right),
\end{eqnarray*}
where the expectations are taken with respect to $Z$. Then
\begin{eqnarray*}
&&\log J_1\\
&=&-\frac{\sigma }{8u^{3}}\sum_{ijkl}\partial^4
_{ijkl}C(0)y_{i}y_{j}y_{k}y_{l}-\frac{1}{6}u^{-3/2}\sum_{iikl}3\partial^4
_{iikl}C(0)u^{-1/2}y_{k}y_{l} \\
&&+\frac{1}{24}u^{-1}\sum_{iikl}6\partial^4 _{iikl}C(0)u^{-1}y_{k}y_{l}+\frac{1%
}{8\sigma u}\sum_{iiii}\partial^4 _{iiii}C(0)+o\left(u^{-1}\right) \\
&=&-\frac{\sigma }{8u^{3}}\sum_{ijkl}\partial^4
_{ijkl}C(0)y_{i}y_{j}y_{k}y_{l}-\frac{1}{4u^{2}}\sum_{iikl}\partial^4
_{iikl}C(0)u^{-1/2}y_{k}y_{l}+\frac{1}{8\sigma
u}\sum_{iiii}\partial^4
_{iiii}C(0)+o\left(u^{-1}\right) \\
&=&-\frac{\sigma }{8u^{3}}Y^{\top }\mu
_{22}Y-\frac{1}{4u^{2}}Y^{\top }\mu
_{22}\mathbf 1+\frac{1}{8\sigma u}\sum_{iiii}\partial^4 _{iiii}C(0) \\
&=&-\frac{\sigma }{8u}(u^{-1}Y-\mathbf 1/\sigma )^{\top }\mu _{22}(u^{-1}Y-\mathbf 1/\sigma
)+\frac{1}{8\sigma u}\mathbf 1^{\top }\mu _{22}\mathbf 1+\frac{1}{8\sigma u}%
\sum_{iiii}\partial^4_{iiii}C(0)+o\left(u^{-1}\right).
\end{eqnarray*}
\end{proof}

\begin{proof}[Proof of Lemma \protect\ref{LemI4}]
Similar to the derivations of the leading term $I_{2}$ in Part 2 of Theorem \ref{ThmG} proof, such as \eqref{inter} and \eqref{cont}, we can obtain that
\begin{eqnarray*}
I_{3} &=&O(1)e^{u\sup |g(t)|}\det
(\tilde{u}I-z)^{-1/2}\tilde{u}^{-d/2}\\
&&\times\exp \left\{(\tilde{u}+\zeta _{u})\left(f(\tau
)+\frac{1}{2}\tilde{y}^{\top }(
\tilde{u}I-\tilde{\zz})^{-1}\tilde{y}\right)\right\}  \notag \\
&&\times
\int_{|\tilde{u}^{-1/2}(\tilde{u}I-\tilde{\mathbf{z}})^{-1/2}s|>\lambda
_{u}}\exp \left\{-\frac{1}{2}\Big\vert s-(\tilde{u}+\zeta _{u})^{1/2}(\tilde{u}%
I- \tilde{\mathbf{z}})^{-1/2}\tilde{y}\Big\vert ^{2}\right\}  \notag \\
&&\times \exp \biggr\{(\tilde{u}+\zeta
_{u})g_{3}\left((\tilde{u}+\zeta _{u})^{-1/2}(
\tilde{u}I-\tilde{\zz})^{-1/2}s\right)\notag\\
&&~~~~~~~~+\tilde{u}(\tilde{u}+\zeta _{u})C_{4}\left((\tilde{%
u }+\zeta
_{u})^{-1/2}(\tilde{u}I-\tilde{\zz})^{-1/2}s\right)\biggr\}ds \notag
\end{eqnarray*}
Note that $|y|\leq u^{1/2+\varepsilon }\ll u^{1/2+\delta }=u\lambda
_{u}$. Therefore, there exists a $\delta ^{\ast }>0$ such that the
integral on the right-hand-side of the above display is
\begin{eqnarray*}
&&\int_{|\tilde{u}^{-1/2}(\tilde{u}I-\tilde{\mathbf{z}})^{-1/2}s|>\lambda
_{u}}\exp \left\{-\frac{1}{2}\left\vert s-(\tilde{u}+\zeta
_{u})^{1/2}(\tilde{u}I-\tilde{\zz}
)^{-1/2}\tilde{y}\right\vert ^{2}\right\} \\
&&\times \exp \biggr\{(\tilde{u}+\zeta
_{u})g_{3}\left((\tilde{u}+\zeta _{u})^{-1/2}(
\tilde{u}I-\tilde{\zz})^{-1/2}s\right)\notag\\
&&~~~~~~~~+\tilde{u}(\tilde{u}+\zeta _{u})C_{4}\left((\tilde{%
u }+\zeta _{u})^{-1/2}(\tilde{u}I-\tilde{\zz})^{-1/2}s\right)\biggr\}ds \\
&=&O\left(e^{-\delta ^{\ast }u^{1+2\delta }}\right).
\end{eqnarray*}
\end{proof}

\begin{proof}[Proof of Lemma \protect\ref{LemGamma}]
To simplify the notation, let $\tau =0$. For other values of $\tau
$, the derivation is completely analogous. Recall that
\begin{eqnarray*}
\gamma _{u}(-x) &=&E\left[ \left. \frac{1}{\kappa +O(e^{-\delta
^{\ast }u^{1+2\delta }}e^{u\sup |g(t)|})};-x>\xi _{u}\right\vert
f(\tau ),\partial
f(\tau ),\partial ^{2}f(\tau )\right] \\
&\leq &E\left[ \kappa ^{-1};-x>\xi _{u}|f(\tau ),\partial f(\tau
),\partial ^{2}f(\tau )\right] .
\end{eqnarray*}%
We use the notation $E_{w,\tilde{y},\tilde{z}}[\cdot ]$ to denote
$E\left[\cdot |f(\tau )=w,\partial f(\tau )=\tilde{y},\partial
^{2}f(\tau )=\tilde{z}\right]$ and $P_{w,\tilde{y},\tilde{z}}$ to
denote the conditional probability.

Note that
\begin{equation*}
e^{-\xi _{u}/u}=E\left[ \exp \left\{
g\left((uI-\zz)^{-1/2}\tilde{S}\right)\right\}\right],
%;(uI-\zz)^{-1/2}\tilde{S}\in T .
\end{equation*}%
where $\tilde{S}$ is a random variable  with density
proportional to (\ref{den}). Therefore,
\begin{equation*}
-x>\xi _{u},
\end{equation*}%
if and only if
\begin{equation*}
A_{1}+A_{2}>e^{x/u},
\end{equation*}%
where
\begin{eqnarray*}
A_{1}
&=&E\left[\exp\left\{g\left((uI-\zz)^{-1/2}\tilde{S}\right)\right\};
|\tilde{S}|\leq u^{\delta}\right], \\
A_{2}
&=&E\left[\exp\left\{g\left((uI-\zz)^{-1/2}\tilde{S}\right)\right\};
|\tilde{S}|>u^{\delta},(uI-\zz)^{-1/2}\tilde{S}\in T\right].
\end{eqnarray*}%
Furthermore, we have that
\begin{equation*}
\log A_{1}\leq \sup_{|s|\leq u^{-1/2+\delta }}g(s),
\end{equation*}%
and for $\lambda $ and $\lambda'$ sufficiently small
\begin{eqnarray*}
P_{w,\tilde{y},\tilde{z}}\Big(A_{2} &>&e^{-\lambda u^{2\delta }}\Big) \\
&\leq &P_{w,\tilde{y},\tilde{z}}\left( \sup_{u^{\delta }\leq s\leq \sqrt{u}%
}g(s/\sqrt{u})-\frac{1}{2}s^{2}>-\lambda u^{2\delta }\right) \leq
e^{-\lambda' u^{2}}\leq e^{-\lambda' u^{1/2-\varepsilon }x},
\end{eqnarray*}%
for all $0<x<u^{3/2+\varepsilon }$. Therefore,
\begin{eqnarray*}
&&E_{w,\tilde{y},\tilde{z}}\left[ \kappa
^{-1};~A_{1}+A_{2}>e^{x/u},0<A_{2}<e^{-\lambda u^{2\delta }}\right] \\
&=&E_{w,\tilde{y},\tilde{z}}\left[ \kappa ^{-1};~u\cdot
\sup_{|s|\leq
u^{-1/2+\delta }}g(s)+O\left(u\cdot e^{-\lambda u^{2\delta }}\right)>x\right] \\
&\leq &E_{w,\tilde{y},\tilde{z}}\left[ \exp\left\{ O\left(u\cdot
\sup_{|s|\leq u^{-1/2+\delta }}|g(s)|\right)\right\};~u\cdot
\sup_{|s|\leq u^{-1/2+\delta }}g(s)+O\left(u\cdot e^{-\lambda
u^{2\delta }}\right)>x\right].
\end{eqnarray*}%
By the Borell-TIS inequality (Lemma \ref{LemBorel}), we have that
\begin{equation*}
P\left( u\cdot \sup_{|s|\leq u^{-1/2+\delta }}g(s)>x\right) \leq e^{-\lambda''
u^{1-6\delta }x^{2}}.
\end{equation*}%
Therefore, with $\delta ^{\ast }$ small enough, we have that
\begin{equation*}
E_{w,\tilde{y},\tilde{z}}\left[ \kappa
^{-1};A_{1}+A_{2}>e^{x/u},0<A_{2}<e^{-\lambda u^{2\delta }}\right]
\leq e^{-\delta ^{\ast }u^{1-6\delta }x^{2}}.
\end{equation*}%

For the remainder term,
\begin{eqnarray*}
&&E_{w,\tilde{y},\tilde{z}}\left[ \kappa
^{-1};~A_{1}+A_{2}>e^{x/u},A_{2}\geq
e^{-\lambda u^{2\delta }}\right] \\
&\leq &E_{w,\tilde{y},\tilde{z}}\left[ \kappa ^{-1};~A_{2}\geq
e^{-\lambda
u^{2\delta }}\right] \\
&\leq &E_{w,\tilde{y},\tilde{z}}\left[ \kappa ^{-1};\sup_{u^{\delta
}\leq
s\leq \sqrt{u}}g(s/\sqrt{u})-\frac{1}{2}s^{2}>-\lambda u^{2\delta }\right] \\
&\leq &e^{-\lambda' u^{2}} E_{w,\tilde{y},\tilde{z}}\left[ \kappa
^{-1}\biggr|\sup_{u^{\delta }\leq s\leq
\sqrt{u}}g(s/\sqrt{u})-\frac{1}{2}s^{2}>-\lambda u^{2\delta }\right]
\\
&=&e^{-\lambda' u^{2}}E_{w,\tilde{y},\tilde{z}}\left[ \kappa
^{-1}\biggr|\sup_{s\in T,|s|>u^{-1/2+\delta
}}g(s)-\frac{u}{2}s^{2}>-\lambda u^{2\delta }\right].
\end{eqnarray*}%
Let $s_{\ast }$ be the global maximum of $g(s)-\frac{u}{2}s^{2}$ in
the region $\left\{s: s\in T \hbox{ and } |s|>u^{-1/2+\delta
}\right\}$. In what follows, we provide an
upper bound of $E\left[ \kappa ^{-1}\big|\sup_{s\in T,|s|>u^{-1/2+\delta }}g(s)-\frac{u}{2}s^{2}>-\lambda u^{2\delta }\right] .$ Note that for each
$s_{\ast }\in T,|s_{\ast }|>u^{-1/2+\delta }$, we proceed to
describing the conditional distribution of $g(t)$ given $g(s_{\ast
})-\frac{u}{2}s_{\ast }^{2}>-\lambda u^{2\delta }$. Recall that
$g(t)$ is the remainder of the $f(t)$ after the first two order
expansion. Therefore, we have
\begin{equation*}
g(t)=\sum_{i,j,k}Z_{i,j,k}t_{i}t_{j}t_{k}+o_{p}(|t|^{3}),
\end{equation*}%
where $Z_{i,j,k}$ jointly follows a multivariate Gaussian random
distribution with  mean zero and nondegenerate covariance. Then,
 $g(t)$ approximately takes a form of a cubic polynomial.
Therefore, given $g(s_{\ast })-\frac{u}{2}s_{\ast }^{2}>-\lambda
u^{2\delta }$, we have approximation
\begin{equation*}
g(t)=O_p\left(u|s_{\ast }|^{2}\frac{|t|^{3}}{|s_{\ast
}|^{3}}\right)=O_p\left(u|t|^{3}|s_{\ast }|^{-1}\right).
\end{equation*}%
We write $X_n= O(a_n)$ if $M_n(\theta) = E(e^{\theta X_n/a_n})$ in bounded uniformly of $\theta$ in a domain around zero. The above result can be obtained by repeatedly using the Borell-TIS inequality (Lemma \ref{LemBorel}) and we omit the details. Thus, for $|\tilde{y}|\leq u^{1/2+\varepsilon }$ and given $g(s_{\ast})-\frac{u}{2}s_{\ast }^{2}>-\lambda u^{2\delta} $, we have the approximation
\begin{equation*}
ug\left(u^{-1}\tilde{y}\right)=O_p\left(u^{1+3\varepsilon
-\delta}\right).
\end{equation*}%
We choose $3\varepsilon <\delta $ and some $\zeta >0$ sufficiently large so
that%
\begin{eqnarray*}
\kappa &=&\int_{|\tilde{u}^{-1/2}(\tilde{u}I-\tilde{\mathbf{z}})^{-1/2}s|<\lambda
_{u}}\exp \left\{-\frac{1}{2}\left\vert s-(\tilde{u}+\zeta
_{u})^{1/2}(\tilde{u}I-\tilde{\zz}
)^{-1/2}\tilde{y}\right\vert ^{2}\right\}ds \\
&&\times E\exp \left\{(\tilde{u}+\zeta_{u})g\left((\tilde{u}+\zeta
_{u})^{-1/2}(\tilde{u}I-\tilde{\zz}
)^{-1/2}S\right)\right\} \\
&\geq & \Theta(1)\int_{|s|\leq u^{1/2+\delta }}\exp \left\{ -\zeta |s-\tilde{y}
|^{2}+O(ug(s/u))\right\} ds \\
&\geq &\Theta(1)\int_{|s-\tilde{y}|\leq 1}\exp \left\{ -\zeta |s-\tilde{y}
|^{2}+O(ug(s/u))\right\} ds \\
&\geq &\Theta(1)\int_{|s-\tilde{y}|\leq 1}\exp \left\{ -\zeta |s-\tilde{y}
|^{2}+O(u^{1+3\varepsilon -\delta })\right\} ds \\
&\geq & e^{O(u)},
\end{eqnarray*}%
where the $O(u)$ in the last step could be negative. Then,
\begin{equation}\label{below}
E_{w,\tilde{y},\tilde{z}}\left[ \kappa ^{-1}\biggr|\sup_{s\in
T,|s|>u^{-1/2+\delta }}g(s)-\frac{u}{2}s^{2}>-\lambda u^{2\delta
}\right] \leq e^{O(u)}.
\end{equation}%
Therefore, for all $0<x<u^{3/2+\varepsilon}$, we obtain that
\begin{equation*}
 e^{-\lambda u^{2}}E_{w,\tilde{y},\tilde{z}}\left[ \kappa ^{-1}\biggr|\sup_{s\in
T,|s|>u^{-1/2+\delta }}g(s)-\frac{u}{2}s^{2}>-\lambda u^{2\delta
}\right]\leq e^{-\lambda u^{2}+O(u)}\leq e^{-u^{\delta ^{\ast }}x}.
\end{equation*}
\end{proof}

\begin{proof}[Proof of Lemma \ref{LemBoundary}]
%We now consider that $\tau $ is within $u^{-1/2+\delta ^{\prime }}$ distance
%from the boundary of $T$. Basically, we establish that the $\Lambda(\tau)$ in \eqref{lambda} is of the same order as the approximation in \eqref{asy}. In particular, we show that, for such a $\tau $,
%\begin{equation*}
%E_{\tau }^{Q}\left[ \frac{dP}{dQ};\int_{T}e^{\mu (t)+\sigma f(t)}dt>b,%
%\mathcal{L}_{Q}\right] =mes(T)e^{u^{2}/2}E\left[ K^{-1};\mathcal{E}_{b},%
%\mathcal{L}\right]= O(1)u^{d-1}e^{-\frac 1 2 (u-\mu _{\sigma
%}(\tau ))^{2}}.
%\end{equation*}%

Similar to the interior case, we proceed by conditioning on $(f(\tau ),\partial
f(\tau ),\partial ^{2}f(\tau ))$ and consider

\begin{equation*}
mes(T)e^{u^{2}/2}E\left[ \left. K^{-1};\mathcal{E}_{b},\mathcal{L}%
\right\vert f(\tau )=w,\partial f(\tau )=\tilde{y},\partial ^{2}f(\tau )=%
\tilde{z}\right] .
\end{equation*}%

We perform a similar analysis as in Parts 1 and 2 in the proof of Theorem \ref{ThmG}. Therefore, we will omit most details. With a similar analysis as $I_1$ in \eqref{l4} and \eqref{ineq},  we obtain that
\begin{equation*}
\int_{T}e^{\sigma \left\{ f(t)+(u-\mu _{\sigma }(\tau ))C(t-\tau )+\mu
_{\sigma }(t)\right\} }dt>b
\end{equation*}%
if and only if
\begin{equation*}
A+\log \int_{(uI-\mathbf{z})^{-\frac{1}{2}}s+\tau \in T}\left( \frac{\sigma }{2\pi }%
\right) ^{d/2}e^{-\frac{\sigma
}{2}|s-(uI-\mathbf{z})^{-1/2}y|^{2}}ds>u^{-1}\xi _{u}
\end{equation*}%
where $A$ and $\xi _{u}$ are as defined in (\ref{A}) and (\ref{xi}). Note
that if $\tau $ is in the interior, the second term on the left-hand-side of
the above inequality is $o(u^{-1})$; if $\tau $ is close to the boundary,
this term may take a negative value. This is because part of the important integration region lies outside of set $T$.

Now we consider the $K$ term. According to
the derivations in Part 2, we obtain that%
\begin{eqnarray*}
K &=&(1+o(1))u^{-d}e^{u^{2}-u\mu _{\sigma }(\tau )+\frac{1}{2}\mu^{2}_{\sigma
}(\tau )} \\
&&\times \exp \biggr\{\frac{\tilde{u}}{\sigma }A-\tilde{u}B-\frac{1}{2\sigma }%
Tr(\Delta f(\tau )+\mu _{\sigma }(\tau )I+\Delta \mu _{\sigma }(\tau ))-%
\tilde{y}^{\top }\partial \mu _{\sigma }(\tau )-\frac{1}{2}|\partial \mu
_{\sigma }(\tau )|^{2} \\
&&+u^{-1/2+\delta}O\left(|y|^{2}+|z|^{2}+f(\tau )\right)\biggr\} \\
&&\times \int_{\tilde{u}^{-1/2}(\tilde{u}I-\tilde{\mathbf{z}})^{-1/2}s+\tau \in T}\exp \left\{-%
\frac{1}{2}\Big\vert s-(\tilde{u}+\zeta _{u})^{1/2}(\tilde{u}I-\tilde{\zz}%
)^{-1/2}\tilde{y}\Big\vert ^{2}\right\}ds \\
&&\times E\exp \left\{(\tilde{u}+\zeta _{u})g\left((\tilde{u}+\zeta _{u})^{-1/2}(\tilde{u}I-\tilde{\zz}%
)^{-1/2}S^{\prime }\right)\right\}.
\end{eqnarray*}%
Therefore,%
\begin{eqnarray*}
&&E_{\tau }^{Q}\left[ \frac{dP}{dQ};\int_{T}e^{\mu (t)+\sigma f(t)}dt>b,%
\mathcal{L}_{Q}\right]  \\
&=&O(1)\int e^{u^{2}/2}E\left[ \left. K^{-1};\mathcal{E}_{b},\mathcal{L}%
\right\vert f(\tau )=w,\partial f(\tau )=\tilde{y},\partial ^{2}f(\tau )=%
\tilde{z}\right] h(w,\tilde{y},\tilde{z})dwd\tilde{y}d\tilde{z}.
\end{eqnarray*}%
With similar derivations as in Part 3 and letting $E_{w,\tilde y, \tilde z} (\cdot ) = E(\cdot |w,\tilde y, \tilde z) $, we obtain that
\begin{eqnarray}
&&E_{\tau }^{Q}\left[ \frac{dP}{dQ};\int_{T}e^{\mu (t)+\sigma f(t)}dt>b,%
\mathcal{L}_{Q}\right]   \label{Integrand} \\
&=&O(1)u^{d}e^{-\frac{1}{2}(u-\mu _{\sigma }(\tau ))^{2}}  \notag \\
&&\int_{\mathcal{L}}\left( \int_{\tilde{u}^{-1/2}(\tilde{u}I-\tilde{\mathbf{z%
}})^{-1/2}s+\tau \in T}\exp\left\{-\frac{1}{2}\left\vert s-(\tilde{u}+\zeta
_{u})^{1/2}(\tilde{u}I-\tilde{\zz})^{-1/2}\tilde{y}\right\vert
^{2}\right\}ds\right) ^{-1}  \notag \\
&&\times E_{w,\tilde{y},\tilde{z}}\biggr[\frac{1}{E\exp \left\{ (\tilde{u}%
+\zeta _{u})g\left( (\tilde{u}+\zeta _{u})^{-1/2}(\tilde{u}I-\tilde{\zz}%
)^{-1/2}S^{\prime }\right) \right\} };  \notag \\
&&~~~~~~~~~~~A+\log \int_{(uI-\mathbf{z})^{-\frac{1}{2}}s+\tau \in T}\left(
\frac{\sigma }{2\pi }\right) ^{d/2}e^{-\frac{\sigma }{2}|s-(uI-\mathbf{z}%
)^{-1/2}y|^{2}}ds>u^{-1}\xi _{u}\biggr]  \notag \\
&&\times \exp \left\{ -\frac{\tilde{u}}{\sigma }A+u^{-1/2+\delta }O(|\tilde{y%
}|^{2}+|\tilde{z}|^{2}+|w|)\right\}   \notag \\
&&\times \exp \biggr\{-\frac{1}{2}\biggr[|\tilde{y}-\partial \mu _{\sigma
}(\tau )|^{2}+\frac{(A/\sigma +O(u^{-1/2+\varepsilon }|\tilde{y}|)+o(1)-\mu
_{20}\mu _{22}^{-1}\tilde{z})^{2}}{1-\mu _{20}\mu _{22}^{-1}\mu _{02}}
\notag \\
&&~~~~~~~~+\left\vert \mu _{22}^{-1/2}\tilde{z}-\mu _{22}^{1/2}\frac{\mathbf{%
1}}{2\sigma }\right\vert ^{2}\biggr]\biggr\}dwd\tilde{y}d\tilde{z}  \notag
\end{eqnarray}
There are two terms whose analyses are different from those of the interior case:
\begin{equation}\label{K}
\int_{\tilde{u}^{-1/2}(\tilde{u}I-\tilde{\mathbf{z}})^{-1/2}s+\tau \in T}\exp \left\{-\frac{1}{2}%
\left\vert s-(\tilde{u}+\zeta _{u})^{1/2}(\tilde{u}I-\tilde{\zz})^{-1/2}\tilde{%
y}\right\vert ^{2}\right\}ds
\end{equation}%
and%
\begin{equation}\label{EE}
\int_{(uI-\mathbf{z})^{-\frac{1}{2}}s+\tau \in T}\left( \frac{\sigma }{2\pi }\right)
^{d/2}e^{-\frac{\sigma }{2}|s-(uI-\mathbf{z})^{-1/2}y|^{2}}ds.
\end{equation}
This is mainly because the main integration region may lie beyond the region $T$.
We bound the above two integrals in the following two cases.

\begin{description}
\item[Case 1.] We consider the set $R_{1}=\{\tilde{y}:\tau +(\tilde u I-\tilde{\zz} )^{-1}\tilde{y}%
\in T\}$.

\item[Case 2.] We consider the set $R_{2}=\{\tilde{y}:\tau +(\tilde u I-\tilde{\zz} )^{-1}\tilde{y}%
\notin T\}$.
\end{description}

For the first case, $\tau +(\tilde u I- \tilde {\zz})^{-1}\tilde{y} \in T$. Note that the
integrands of \eqref{K} and \eqref{EE}
are approximately Gaussian densities. Since $\tau +(\tilde u I-\tilde {\zz})^{-1}\tilde{y}\in T$, there exists $\delta _{0}$ so that
\begin{equation*}
\int_{\tilde{u}^{-1/2}(\tilde{u}I-\tilde{\mathbf{z}})^{-1/2}s+\tau \in T}\exp \left\{-\frac{1}{2}%
\left\vert s-(\tilde{u}+\zeta _{u})^{1/2}(\tilde{u}I-\tilde{\zz})^{-1/2}\tilde{%
y}\right\vert ^{2}\right\}ds\geq \delta _{0}.
\end{equation*}%
In addition, $\eqref{EE}\leq 1$.
Therefore,%
\begin{eqnarray*}
&&e^{u^{2}/2}E\left[ K^{-1};\mathcal{E}_{b},\mathcal{L,}R_{1}\right]  \\
&\leq &O(1)u^{d}e^{-\frac{1}{2}(u-\partial \mu _{\sigma }(\tau
))^{2}} \int_{\mathcal{L}}E_{w,\tilde{y},\tilde{z}}\left[
\frac{1}{Ee^{(\tilde{
u}+\zeta _{u})g((\tilde{%
u}+\zeta _{u})^{-1/2}(\tilde{u}I-\tilde{\mathbf{z}})^{-1/2}S^{\prime
})}};A>u^{-1}\xi
_{u}\right]  \\
&&\times \exp \left\{ -\frac{\tilde{u}}{\sigma }A+u^{-1/2+\delta }O(|%
\tilde{y}|^{2}+|\tilde{z}|^{2}+|w|)\right\}  \\
&&\times \exp \biggr\{ -\frac{1}{2}\biggr[ |\tilde{y}-\partial \mu
_{\sigma }(\tau )|^{2}+\frac{(A/\sigma +O(u^{-1/2+\varepsilon
}|\tilde{y}|)+o(1)-\mu _{20}\mu _{22}^{-1}\tilde{z})^{2}}{1-\mu
_{20}\mu
_{22}^{-1}\mu _{02}}\notag\\
&&~~~~~~~~~+\left|\mu
_{22}^{-1/2}\tilde{z}-\mu _{22}^{1/2}\frac{\mathbf{1}}{2\sigma} \right|^{2}\biggr] \biggr\} dwd\tilde{y}d\tilde{z} \\
&=&O(1)u^{d-1}e^{-\frac{1}{2}(u- \mu _{\sigma }(\tau ))^{2}}.
\end{eqnarray*}%
The last step uses similar arguments as in those in Part 3, such as dominated convergence theorem. Thus, we conclude Case 1.

For the second case $\tau +(\tilde u I-\tilde{\zz} )^{-1}\tilde{y}\notin T$, since $\tau \in T$ and $T$ has piecewise smooth boundary, the integral in \eqref{K} has lower
bound
\begin{eqnarray*}
&&\int_{\tilde{u}^{-1/2}(\tilde{u}I-\tilde{\mathbf{z}})^{-1/2}s+\tau \in T}\exp \left\{-\frac{1}{2}%
\left\vert s-(\tilde{u}+\zeta _{u})^{1/2}(\tilde{u}I-\tilde{\zz})^{-1/2}\tilde{%
y}\right\vert ^{2}\right\}ds \\
&=&\Theta (1)\int_{t+\tau \in T}u^{d}\exp \left\{ -\frac{\tilde{u}%
^{2}(1+o(u^{-1/2+\varepsilon }))}{2}\left\vert t-\tilde{u}^{-1}\tilde{y}%
+O(u^{-3/2+\varepsilon })\right\vert ^{2}\right\} dt \\
&=&\Theta (1)P\left( Z>(1+O(u^{-1/2+\varepsilon}))|\tilde{y}|\right)  \\
&=&\min \left\{ \Theta (1)\frac{1}{|\tilde{y}|}e^{-\frac{1+O(u^{-1/2+\varepsilon})}{2}|\tilde{y}%
|^{2}},1\right\} .
\end{eqnarray*}%
We insert this result back to (\ref{Integrand}) and obtain that%
\begin{eqnarray*}
&&e^{u^{2}/2}E\left[ K^{-1};\mathcal{E}_{b},\mathcal{L,}R_{2}\right]  \\
&\leq &O(1)u^{d}e^{-\frac{1}{2}(u-\mu _{\sigma }(\tau ))^{2}}\int_{\mathcal{%
L,}R_{2}}|\tilde{y}|\exp \{\frac{1}{2}|\tilde{y}|^{2}\} \\
&&\times E_{w,\tilde{y},\tilde{z}}\biggr[\frac{1}{Ee^{(\tilde{u}+\zeta
_{u})g((\tilde{u}+\zeta _{u})^{-1/2}(\tilde{u}I-\tilde{\mathbf{z}}%
)^{-1/2}S^{\prime })}}; \\
&&~~~~A+\log \int_{(uI-\mathbf{z})^{-\frac{1}{2}}s+\tau \in T}\left( \frac{%
\sigma }{2\pi }\right) ^{d/2}e^{-\frac{\sigma }{2}|s-(uI-\mathbf{z}%
)^{-1/2}y|^{2}}ds>u^{-1}\xi _{u}\biggr] \\
&&\times \exp \left\{ -\frac{\tilde{u}}{\sigma }A+u^{-1/2+\delta }O(|\tilde{y%
}|^{2}+|\tilde{z}|^{2}+|w|)\right\}  \\
&&\times \exp \biggr\{-\frac{1}{2}\biggr[|\tilde{y}-\partial \mu _{\sigma
}(\tau )|^{2}+\frac{(A/\sigma +O(u^{-1/2+\varepsilon }|\tilde{y}|)+o(1)-\mu
_{20}\mu _{22}^{-1}\tilde{z})^{2}}{1-\mu _{20}\mu _{22}^{-1}\mu _{02}} \\
&&~~~~~~~~~~~+\left\vert \mu _{22}^{-1/2}\tilde{z}-\mu _{22}^{1/2}\frac{%
\mathbf{1}}{2\sigma }\right\vert ^{2}\biggr]\biggr\}dwd\tilde{y}d\tilde{z}.
\end{eqnarray*}
Using $|\tilde y| ^2$ to cancel the square term of $|\tilde y - \partial \mu_\sigma(\tau)|^2$ and the fact that $|\tilde y|\leq u^{1/2+\varepsilon}$, we obtain that
\begin{eqnarray*}
&&O(1)\tilde u ^de^{-\frac{1}{2}(u-\mu _{\sigma }(\tau ))^{2}}\\
&&\int_{\mathcal{L,}R_{2}}E_{w,\tilde{y},\tilde{z}}\biggr[
\frac{1}{Ee^{(\tilde{ u}+\zeta _{u})g((
\tilde{u}+\zeta _{u})^{-1/2}(\tilde{u}I-\tilde{\mathbf{z}})^{-1/2}S^{\prime })}}%
;\notag\\
&&~~~~A+\log \int_{(uI-\mathbf{z})^{-\frac{1}{2}}s+\tau \in T}\left( \frac{\sigma }{2\pi }%
\right) ^{d/2}e^{-\frac{\sigma }{2}|s-(uI-\mathbf{z})^{-1/2}y|^{2}}ds>u^{-1}\xi
_{u}\biggr] \\
&&\times \exp \left\{ -\frac{\tilde{u}}{\sigma }A+O(u^{1/2+2\delta
})\right\}  \\
&&\times \exp \biggr\{ -\frac{1}{2}\biggr[ \frac{(A/\sigma
+O(u^{-1/2+\varepsilon }|\tilde{y}|)+o(1)-\mu _{20}\mu _{22}^{-1}\tilde{z})^{2}}{%
1-\mu _{20}\mu _{22}^{-1}\mu _{02}}+\left|\mu
_{22}^{-1/2}\tilde{z}-\mu
_{22}^{1/2}\frac{\mathbf{1}}{2\sigma}\right|^{2}\biggr] \biggr\}
dwd\tilde{y}d\tilde{z}.
\end{eqnarray*}%
We now proceed to handle the term in \eqref{EE}. Once again, since $\tau + (\tilde u I-\tilde{\zz} )^{-1}\tilde{y}\notin T$, there exists a $c_{0}>0$ such that
\begin{equation*}
\log \int_{(uI-\mathbf{z})^{-\frac{1}{2}}s+\tau \in T}\left( \frac{\sigma }{2\pi }%
\right) ^{d/2}e^{-\frac{\sigma }{2}|s-(uI-\mathbf{z})^{-1/2}y|^{2}}ds\leq -c_{0}.
\end{equation*}%
Therefore,%
\begin{eqnarray*}
&&E_{\tau }^{Q}\left[ \frac{dP}{dQ};\int_{T}e^{\mu (t)+\sigma f(t)}dt>b,%
\mathcal{L}_{Q},R_{2}\right]  \\
&\leq &O(1)u^d e^{-\frac{1}{2}(u- \mu _{\sigma }(\tau ))^{2}} \\
&&\int_{\mathcal{L,}R_{2}}E_{w,\tilde{y},\tilde{z}}\left[ \frac{1}{Ee^{(%
\tilde{u}+\zeta _{u})g((\tilde{u}+\zeta _{u})^{-1/2}(\tilde{u}I-\tilde{\mathbf{z}})^{-1/2}S^{\prime })}}%
;A>c_{0}+u^{-1}\xi _{u}\right]  \\
&&\times \exp \left\{ -\frac{\tilde{u}}{\sigma }A+O(u^{1/2+2\delta
})\right\}  \\
&&\exp \left\{ -\frac{1}{2}\left[ \frac{(A/\sigma +O(u^{-1/2+\varepsilon }|%
\tilde{y}|)+o(1)-\mu _{20}\mu _{22}^{-1}\tilde{z})^{2}}{1-\mu
_{20}\mu _{22}^{-1}\mu _{02}}+\left|\mu _{22}^{-1/2}\tilde{z}-\mu
_{22}^{1/2}\frac{\mathbf{1}}{2\sigma} \right|^{2}\right] \right\}
dwd\tilde{y}d\tilde{z}.
\end{eqnarray*}%
The integral in the above display can be split into two parts, $A > c_0/2$ and $A< c_0 /2$: first
\begin{eqnarray*}
&&\int_{\mathcal{L,}R_{2},A>c_{0}/2}E_{w,\tilde{y},\tilde{z}}\left[ \frac{1%
}{Ee^{( \tilde{u}+\zeta _{u})g((\tilde{u}+\zeta
_{u})^{-1/2}(\tilde{u}I-\tilde{\mathbf{z}})^{-1/2}S^{\prime
})}};A>c_{0}+u^{-1}\xi _{u}\right]  \\
&&\times \exp \left\{ -\frac{\tilde{u}}{\sigma }A+O(u^{1/2+2\delta
})\right\}  \\
&&\exp \left\{ -\frac{1}{2}\left[ \frac{(A/\sigma +O(u^{-1/2+\varepsilon }|%
\tilde{y}|)+o(1)-\mu _{20}\mu _{22}^{-1}\tilde{z})^{2}}{1-\mu
_{20}\mu _{22}^{-1}\mu _{02}}+\left|\mu _{22}^{-1/2}\tilde{z}-\mu
_{22}^{1/2}\frac{\mathbf{1}}{2\sigma}
\right|^{2}\right] \right\} dwd\tilde{y}d\tilde{z} \\
&=&O(1)\int_{\mathcal{L,}R_{2},A>c_{0}/2}\exp \left\{ -\frac{c_{0}\tilde{u}}{2\sigma }-\frac{\tilde u}{\sigma}(A-c_0/2)+O(u^{1/2+2\delta })\right\}  \\
&&\exp \left\{ -\frac{1}{2}\left[ \frac{(A/\sigma +O(u^{2\varepsilon })+o(1)-\mu _{20}\mu _{22}^{-1}\tilde{z})^{2}}{1-\mu
_{20}\mu _{22}^{-1}\mu _{02}}+\left|\mu _{22}^{-1/2}\tilde{z}-\mu
_{22}^{1/2}\frac{\mathbf{1}}{2\sigma}\right|^{2}\right] \right\} dwd\tilde{y}d\tilde{z} \\
&=&O(1)\exp\left\{-\frac{c_{0}\tilde{u}}{4\sigma }\right\},
\end{eqnarray*}%
and second (with a similar derivation as in the proof of Lemma \ref{LemGamma})%
\begin{eqnarray*}
&&\int_{\mathcal{L,}R_{2},A\leq
c_{0}/2}E_{w,\tilde{y},\tilde{z}}\left[ \frac{1}{Ee^{(
\tilde{u}+\zeta _{u})g((\tilde{u}+\zeta _{u})^{-1/2}(\tilde{u}I-\tilde{\mathbf{z}}%
)^{-1/2}S^{\prime })}};A>c_{0}+u^{-1}\xi _{u}\right]  \\
&&\times \exp \left\{ -\frac{\tilde{u}}{\sigma }A+O(u^{1/2+2\delta
})\right\}  \\
&&\exp \left\{ -\frac{1}{2}\left[ \frac{(A/\sigma +O(u^{-1/2+\varepsilon }|%
\tilde{y}|)+o(1)-\mu _{20}\mu _{22}^{-1}\tilde{z})^{2}}{1-\mu
_{20}\mu _{22}^{-1}\mu _{02}}+\left|\mu _{22}^{-1/2}\tilde{z}-\mu
_{22}^{1/2}\frac{\mathbf{1}}{2\sigma}\right|^{2}\right] \right\} dwd\tilde{y}d\tilde{z} \\
%&\leq &\int_{\mathcal{L,}R_{2},A\leq
%c_{0}/2}E_{w,\tilde{y},\tilde{z}}\left[ \frac{1}{Ee^{(\tilde{u}+\zeta _{u})g((\tilde{u}+\zeta _{u})^{-1/2}(\tilde{u}I-\tilde{\mathbf{z}}%
%)^{-1/2}S^{\prime })}};-u^{-1}\xi _{u}>\frac{c_{0}}{2}\right]  \\
%&&\times \exp \left\{ -\frac{\tilde{u}}{\sigma }A+O(u^{1/2+2\delta
%})\right\}  \\
%&&\exp \left\{ -\frac{1}{2}\left[ \frac{(A/\sigma +O(u^{-1/2+\varepsilon }|%
%\tilde{y}|)+o(1)-\mu _{20}\mu _{22}^{-1}\tilde{z})^{2}}{1-\mu
%_{20}\mu
%_{22}^{-1}\mu _{02}}+\left|\mu _{22}^{-1/2}\tilde{z}-\mu _{22}^{1/2}\frac{\mathbf{1}}{2\sigma}\right|^{2}\right] \right\} dwd\tilde{y}d\tilde{z} \\
&\leq &\int_{\mathcal{L,}R_{2},A\leq c_{0}/2}\exp \left\{ -\lambda
u^{2}\right\}  \exp \left\{ -\frac{\tilde{u}}{\sigma
}A+O(u^{1/2+2\delta
})\right\}  \\
&&\exp \left\{ -\frac{1}{2}\left[ \frac{(A/\sigma +O(u^{-1/2+\varepsilon }|%
\tilde{y}|)+o(1)-\mu _{20}\mu _{22}^{-1}\tilde{z})^{2}}{1-\mu
_{20}\mu _{22}^{-1}\mu _{02}}+\left|\mu _{22}^{-1/2}\tilde{z}-\mu
_{22}^{1/2}\frac{\mathbf{1}}{2\sigma}
\right|^{2}\right] \right\} dwd\tilde{y}d\tilde{z} \\
&=&o(u^{-1}).
\end{eqnarray*}%
The last inequality results from the fact that
\begin{equation*}
E_{w,\tilde{y},\tilde{z}}\left[ \frac{1}{Ee^{(\tilde{u}+\zeta
_{u})g((\tilde{u}+\zeta
_{u})^{-1/2}(\tilde{u}I-\tilde{\mathbf{z}})^{-1/2}S^{\prime })}};~-u^{-1}\xi _{u}>\frac{%
c_{0}}{2}\right] = O(1)\exp \left\{ -\lambda u^{2}\right\},
\end{equation*}
 which can be derived
by the bounds of $\xi _{u}$ straightforwardly. Thus, when $\tau $ is within $%
u^{-1/2+\delta ^{\prime }}$ distance away from the boundary of $T$, we
obtain that%
\begin{eqnarray}
&&E_{\tau }^{Q}\left[ \frac{dP}{dQ};\int_{T}e^{\mu (t)+\sigma f(t)}dt>b,%
\mathcal{L}_{Q}\right]   \notag \\
&=&E_{\tau }^{Q}\left[ \frac{dP}{dQ};\int_{T}e^{\mu (t)+\sigma f(t)}dt>b,%
\mathcal{L}_{Q},R_{1}\right] +E_{\tau }^{Q}\left[ \frac{dP}{dQ}%
;\int_{T}e^{\mu (t)+\sigma f(t)}dt>b,\mathcal{L}_{Q},R_{2}\right]
\label{asy1} \notag\\
&=&O(1)u^{d-1}e^{-\frac{1}{2}(u-\mu _{\sigma }(\tau ))^{2}}.
\notag
\end{eqnarray}

\end{proof}

\centerline{\large{\bf B: Simulation}}
%\section{Simulation} \label{SecSim}

\begin{figure}[h]
   \begin{center}
   \includegraphics[width=14cm]{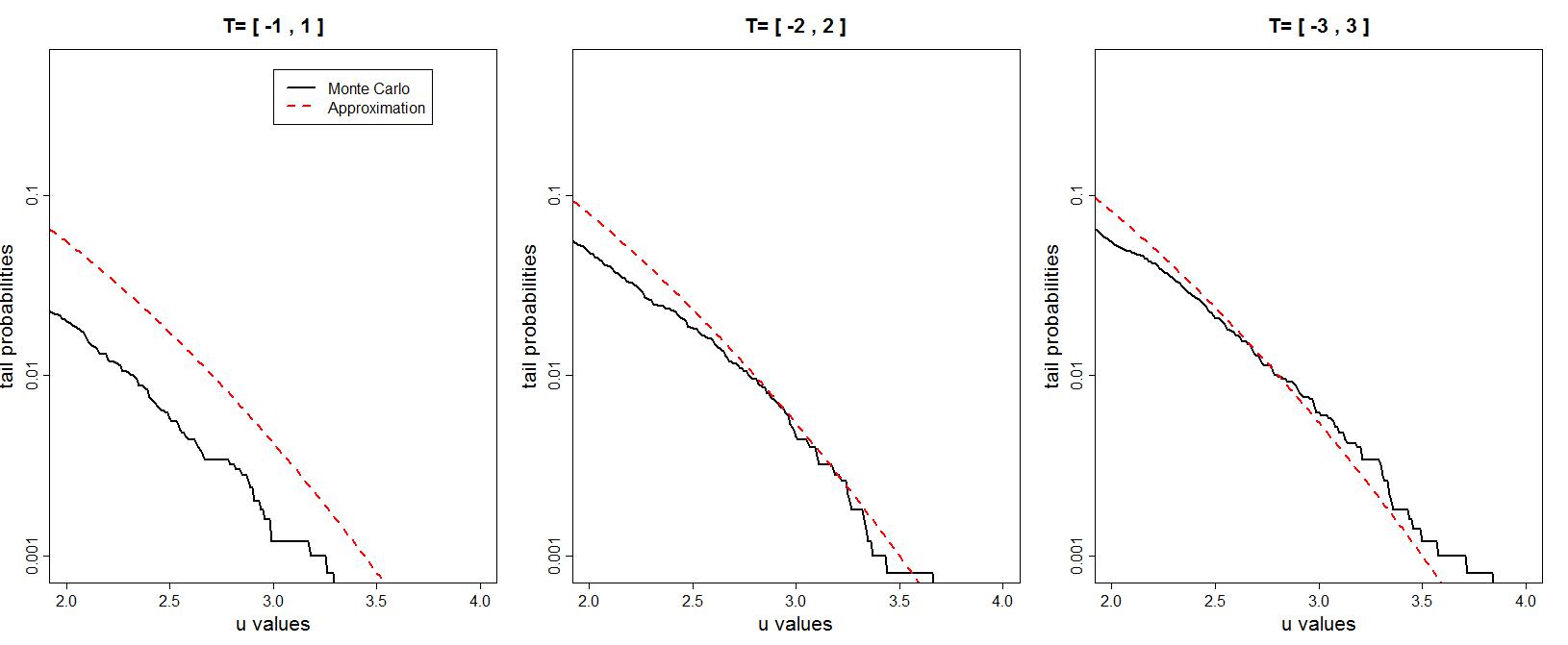}
      \caption{Comparisons of approximations in Theorem \ref{ThmG} and the estimated tail probabilities via 5000 simulations on a log scale for the one-dimensional case.}\label{Fig1}
   \end{center}
   \end{figure}

\begin{figure}[h]
   \begin{center}
   \includegraphics[width=14cm]{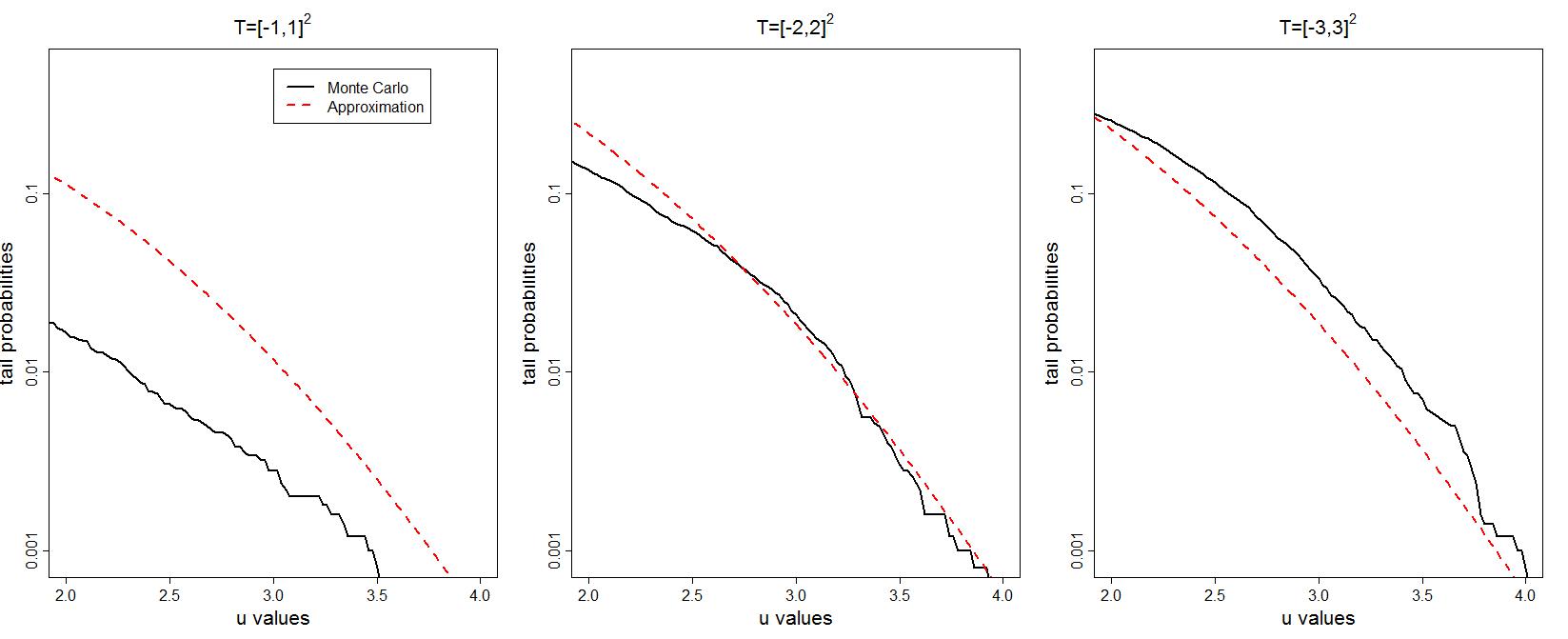}
      \caption{Comparisons of approximations in Theorem \ref{ThmG} and the estimated tail probabilities via 5000 simulations on a log scale for the two-dimensional case.}\label{Fig2}
   \end{center}
   \end{figure}
We conduct a simulation study to exam the accuracy of the approximations. In particular, we consider Gaussian processes living on a domain $T=[-A,A]^d$ with covariance matrix $C(t) = e^{-\frac{|t|^2}{2}}$ and mean function $\mu(t) = - |t|^2/4$.
Figure \ref{Fig1} compares,  on a log scale, the approximations in Theorem \ref{ThmG} with the tail probabilities computed using 5000 simulations for the one-dimensional case, that is, $T=[-A,A]$; Figure \ref{Fig2} is for the two-dimensional case, that is, $T=[-A,A]^2$. The red dash line indicates our approximations and the black solid line indicates the Monte Carlo estimates of the probabilities
$$P\left(\int_T e^{\mu(t)+ f(t)}dt > b\right).$$

One empirical finding is that the approximation sometimes overestimates the tail probability especially when the region $T$ is relatively small (compared with $u^{-1/2}$). This is because the asymptotic approximation is derived based on the heuristics in \eqref{app} so that $\int_T e^{f(t)} dt>b$ is approximately equivalent to $\int_{R^d} e^{uC(t)} dt > b$. However, for not so large a $u$ (or equivalently $b$), for instance $u\leq 4$, the intervals $[-1,1]$ and $[-1,1]^2$ (the first plots of both figures) are not large enough (compared with $u^{-1/2}= \frac 1 2$). Therefore, the approximation in \eqref{app} is not accurate and $\int_T e^{f(t)} dt <\int_{R^d} e^{uC(t)} dt$. This results in an overestimation of the tail. When the region $T$ is larger, such as the second and the third plots in both figures, the approximations are accurate. In many spatial analyses, the region $T$ of interest is reasonably large.
Therefore, we believe that the approximations derived in this paper provide reasonable estimates of the tail probabilities for a wide range of practical analyses.

A practical guide to evaluate the appropriateness of the approximation is as follows. Let $B(t,r) = \{s: |s-t|\leq r\}$ be the $r$-ball around $t$. Define $$r(T) = \sup\{r: \mbox{there exists a $t$ such that } B(t,r)\subset T\}.$$
That is, $r(T)$ is the radius of the largest ball  contained in $T$. Further, define
$$\rho(T)= \sup_{|t|=r(T)} C(t)/C(0).$$
The smaller the $\rho(T)$ is, the larger the region $T$ is (relative to the covariance function).
For instance, for $T=[-1,1]$ and $T=[-1,1]^2$ (the first two plots in both figures), $\rho(T)= e^{-1/2} = 0.61$; for the other regions, the $\rho(T)$'s are all below $0.15$. Therefore, the simulation study supports a practical guide that the approximations developed in the current paper provide good estimates when $\rho(T)< 0.15$.

\end{document}